\documentclass[11pt]{article}
\usepackage{times}
\usepackage{graphicx}
\usepackage{fullpage}
\usepackage{amsmath,amssymb,amsthm}
\usepackage{enumerate}
\usepackage[mathcal]{euscript}
\usepackage{comment}
\usepackage[round]{natbib}

\usepackage[T1]{fontenc}
\usepackage{microtype}
\usepackage{enumitem}

    \numberwithin{equation}{section}

    \usepackage[usenames]{color}
    \definecolor{plum}  {rgb}{.4,0,.4}
    \definecolor{BrickRed} {rgb}{0.6,0,0}
    
    \usepackage{commath}
    \usepackage[normalem]{ulem}

    \usepackage[plainpages=false,pdfpagelabels,colorlinks=true,linkcolor=BrickRed,citecolor=plum]{hyperref}

    \def\ddefloop#1{\ifx\ddefloop#1\else\ddef{#1}\expandafter\ddefloop\fi}

    \def\ddef#1{\expandafter\def\csname c#1\endcsname{\ensuremath{\mathcal{#1}}}}
    \ddefloop ABCDEFGHIJKLMNOPQRSTUVWXYZ\ddefloop

    \def\ddef#1{\expandafter\def\csname s#1\endcsname{\ensuremath{\mathsf{#1}}}}
    \ddefloop ABCDEFGHIJKLMNOPQRSTUVWXYZ\ddefloop

    \def\E{\mathbf{E}}
    \def\PP{\mathbf{P}}

    \def\Reals{\mathbb{R}}
    \def\Naturals{\mathbb{N}}

    \def\argmin{\operatornamewithlimits{arg\,min}}

    \def\deq{:=}

    \def\wh#1{\widehat{#1}}
    \def\bd#1{\boldsymbol{#1}}

	\def\tr{{\mathrm{tr}}}

    \def\1{{\mathbf 1}}

	\def\trn{{\hbox{\it\tiny T}}} 
    \def\eps{\varepsilon}

    \newtheorem{theorem}{Theorem}[section]
    \newtheorem{lemma}{Lemma}[section]

    \newtheorem{assumption}{Assumption}[section]
    \newtheorem{remark}{Remark}[section]

    \usepackage{cleveref}

\begin{document}

    \title{\bf Theoretical guarantees for sampling and inference in generative models with latent diffusions}

    \author{Belinda Tzen\thanks{University of Illinois; e-mail: btzen2@illinois.edu.} \and Maxim Raginsky\thanks{University of Illinois; e-mail: maxim@illinois.edu.}}

    \date{}

	    \maketitle
	
		\begin{abstract} We introduce and study a class of probabilistic generative models, where the latent object is a finite-dimensional diffusion process on a finite time interval and the observed variable is drawn conditionally on the terminal point of the diffusion. We make the following contributions:

	We provide a unified viewpoint on both sampling and variational inference in such generative models through the lens of stochastic control.

	We quantify the expressiveness of diffusion-based generative models. Specifically, we show that one can efficiently sample from a wide class of terminal target distributions by choosing the drift of the latent diffusion from the class of multilayer feedforward neural nets, with the accuracy of sampling measured by the Kullback--Leibler divergence to the target distribution.

	Finally, we present and analyze a scheme for unbiased simulation of generative models with latent diffusions and provide bounds on the variance of the resulting estimators. This scheme can be implemented as a deep generative model with a random number of layers.

		\end{abstract}

    \thispagestyle{empty}
	
		\section{Introduction and informal summary of results}

Recently there has been much interest in using continuous-time processes to analyze discrete-time algorithms and probabilistic models \citep{wibisono2016accelerated,li2017SME,mandt2017SGD,chen18neuralODE,yang2018physical}. In particular, diffusion processes have been examined as a way towards a better understanding of first- and second-order optimization methods, as they afford an analysis of behavior over non-convex landscapes using a rich array of techniques from the mathematical physics literature \citep{li2017SME,raginsky2017non,zhang2017hitting}. Gradient flows and diffusions have also found a role in the analysis of deep neural nets, where they are interpreted as describing the limiting case of infinitely many layers, with each layer being `infinitesimally thin' (e.g., \citet{chen18neuralODE,li2018maximumDP}). As in the case of optimization, continuous-time frameworks enable the use of a different set of tools for studying standard questions of relevance, such as sampling and inference, i.e., forward and backward passes through the network.	
	
In this work, we consider a class of generative models where the latent object  $X = \{X_t\}_{t \in [0,1]}$ is a $d$-dimensional diffusion and the observable object $Y$ is a random element of some space $\sY$:
\begin{subequations}\label{eq:generative_model}
\begin{align}
	\dif X_t &= b(X_t,t;\theta)\dif t + \dif W_t, \qquad X_0 = x \label{eq:latent_diffusion}\\
	Y &\sim q(\cdot|X_1) \label{eq:observation_model}
\end{align}
\end{subequations}
where \eqref{eq:latent_diffusion} is a $d$-dimensional It\^o diffusion process whose drift $b(\cdot,\cdot;\theta)$ is a member of some parametric function class, such as multilayer feedforward neural nets, and \eqref{eq:observation_model} prescribes an observation model for generating $Y$ conditionally on $X_1$. To the best of our knowledge, generative models of this form were first considered by \cite{movellan2002diffusions} as a noisy continuous-time counterpart of recurrent neural nets. More recently, \cite{hashimoto16} and \cite{ryder2018SDE_VI} investigated the use of discrete-time recurrent neural nets to approximate the population dynamics of biological systems that are classically modeled by diffusions. It is natural to view \eqref{eq:generative_model} as a continuum limit of deep generative models introduced by \citet{rezende2014stochbackprop} --- in fact, as we explain in Section~\ref{sec:simulation}, one can simulate a model of the above form using a deep generative model with a \textit{random} number of layers. Alternatively, one can think of \eqref{eq:generative_model} as a \textit{neural stochastic differential equation}, in analogy to the neural ODE framework of \citet{chen18neuralODE}. 

There are three main questions that are natural to ask concerning the usefulness of such models: How expressive can they be? How might one sample from such a diffusion process? How might one perform inference on it? As our first contribution, we provide a unified view of sampling and inference through the lens of stochastic control. In particular, by adding a \textit{control} $u_t$ to the drift of some reference diffusion, one can obtain a desired distribution at $t=1$, and the minimal-cost control that yields exact sampling is given by the so-called F\"ollmer drift \citep{follmer1985reversal,daipra1991reciprocal,lehec2013entropy,eldan2018diffusion}.  Complementarily, we show that any control $u_t$ added to the drift $b(\cdot,t;\theta)$ in \eqref{eq:latent_diffusion} leads to a variational upper bound on the log-likelihood of a given tuple of observations $(y_1,\ldots,y_n)$. Variational inference then reduces to minimizing the expected control cost over a tractable class of controls. While we provide a unifying viewpoint that captures both sampling and inference, we emphasize that this is a synthesis of a number of existing results, and serves as a conceptual underpinning and motivation for our subsequent analysis. Specifically, after establishing that diffusion-based generative models can be effectively worked with, we explore their expressive power vis-\`a-vis neural nets: We show that, if the target density of $X_1$ can be efficiently approximated using a neural net, then the corresponding F\"ollmer drift can also be efficiently approximated by a neural net, such that the terminal law of the diffusion with this approximate drift is $\eps$-close to the target density in Kullback--Leibler divergence. Finally, we investigate unbiased simulation methods for generative models with underlying diffusion processes and provide bounds on the variance of the resulting estimators.

\subsection{Method of analysis:  an overview}

To arrive at the unified perspective of sampling and inference, we begin by formulating a stochastic control problem that captures all of our desiderata: sampling from a target probability law $\mu$ at terminal time $t=1$; a set of tractable controls that might be used to take it there; and an appropriate notion of cost with that captures both the `control effort' and the terminal cost that quantifies the discrepancy between the final probability law and the target measure $\mu$. 

Our first result, stated in Theorem~\ref{thm:log_g_cost}, is an explicit characterization of the value function of this control problem, which has a free-energy interpretation and can be understood from an information-theoretic viewpoint: the Kullback--Leibler divergence between the law of the path of the uncontrolled diffusion and that of the path of the controlled diffusion is the expected total work done by the control. The negative free energy with respect to the uncontrolled process is a lower bound on that of the controlled process after accounting for the work done, and equality is achieved by the optimal control. As pointed out above, this result is a synthesis of a number of existing results, and its main purpose is to motivate the use of controlled diffusions in probabilistic generative modeling.

We next examine the \textit{expressiveness} of these generative models, which refers to their ability to generate samples from a given target distribution for $X_1$ when the observation model $q(\cdot|\cdot)$ in \eqref{eq:observation_model} is fixed. In Theorem~\ref{thm:expressive}, we provide quantitative guarantees for obtaining approximate samples from a given target distribution $\mu$ for $X_1$ when the drift $b$ in \eqref{eq:latent_diffusion} is restricted to be a multilayer feedforward neural net. Specifically, we show that, if the density $f$ of $\mu$ with respect to the standard Gaussian measure on $\Reals^d$ can be efficiently approximated by a feedforward neural net, then the corresponding F\"ollmer drift can also be approximated efficiently by a neural net. Moreover, this approximate F\"ollmer drift yields a diffusion $\{\wh{X}_t\}$, such that $\wh{\mu} = {\rm Law}(\wh{X}_1)$ satisfies $D(\mu||\wh{\mu}) \leq \eps$ for a given accuracy $\eps > 0$. Under some  assumptions on the smoothness of $f$ and $\nabla f$ and on their uniform approximability by neural nets, the proof proceeds as follows: First, we show that the F\"ollmer drift can be approximated by a neural net uniformly over a given compact subset of $\Reals^d$ and for all $t \in[0,1]$. Then, to show that the terminal distribution resulting from this approximation is $\eps$-close to $\mu$ in KL-divergence, we use Girsanov's theorem to relate $D(\mu \| \wh{\mu})$ to the expected squared error between the F\"ollmer drift and its neural-net approximation.

Finally, we discuss the issue of unbiased simulation with the goal of estimating expected values of functions of $X_1$. The standard Euler--Maruyama scheme \citep[Chap.~7]{graham2013simulation} is straightforward, but produces a biased estimator. Typically, one uses Monte Carlo sampling to reduce the variance; if the estimator is biased, then the variance will be reduced by a factor of $N^{1-\delta}$ for some $\delta \in (0,1)$, instead of the optimal reduction by the factor of $N$, for $N$ Monte Carlo runs. One way to obtain an \textit{unbiased} estimator is to employ a \textit{random} discretization of the time interval $[0,1]$, where the sampling times are generated by a point process on the real line. Unbiased simulation schemes of this type have been proposed and analyzed by \citet{bally2015parametrix}, \citet{andersson2017unbiased}, and \citet{henry2017unbiased}. Our final result, Theorem~\ref{thm:unbiased}, builds on the latter work and presents an unbiased, finite-variance simulation scheme. Conceptually, the simulation scheme can be thought of as a deep latent Gaussian model in the sense of \citet{rezende2014stochbackprop}, but with a random number of layers. Unfortunately, the variance of the resulting estimator can exhibit exponential dependence on dimension. We show why this is the case via an analysis of the moment-generating function of the point process used to generate the random mesh and propose alternatives to reduce the variance.

\subsection{Notation} The Euclidean norm of a vector $x \in \Reals^d$ will be denoted by $\|x\|$, the transpose of a vector or a matrix will be indicated by $(\cdot)^\trn$. The $d$-dimensional Euclidean ball of radius $R$ centered at the origin will be denoted by $\sB^d(R)$. The standard Gaussian measure on $\Reals^d$ will be denoted by $\gamma_d$. The Euclidean heat semigroup $Q_t$, $t \ge 0$, acts on measurable functions $f : \Reals^d \to \Reals$ as follows:
\begin{align}\label{eq:heat}
	Q_t f(x) \deq \int_{\Reals^d} f(x+\sqrt{t}z)\gamma_d(\dif z) = \E[f(x+\sqrt{t}Z)], \qquad Z \sim \gamma_d.
\end{align}
A function $g : \Reals^d \times [0,1] \to \Reals$ is of class $C^{2,1}$ if it is twice continuously differentiable in the space variable $x \in \Reals^d$ and once continuously differentiable in the time variable $t \in [0,1]$.

\section{Exact sampling and variational inference: a unified stochastic control viewpoint}
\label{sec:control}

Before addressing the specific questions posed in the Introduction, we aim to demonstrate that both sampling and variational inference in generative models of the form \eqref{eq:generative_model} can be viewed through the lens of stochastic control. We give a brief description of the relevant ideas in Appendix~\ref{app:control}; the book by \cite{fleming1975control} is an excellent and readable reference.

\subsection{A stochastic control problem}

Let $(\Omega,\cF,\{\cF_t\},\PP)$ be a probability space with a complete and right-continuous filtration $\{\cF_t\}$, and let $W = \{W_t\}$ be a standard $d$-dimensional Brownian motion adapted to $\{\cF_t\}$. Consider the It\^o diffusion process
\begin{align}\label{eq:generic_diffusion}
	\dif X_t = b(X_t,t) \dif t + \dif W_t, \qquad t \in [0,1];\, X_0 = x_0
\end{align}
where the drift $b : \Reals^d \times [0,1] \to \Reals^d$ is sufficiently well-behaved (say, bounded and Lipschitz). Then the process $\{X_t\}$ admits a \textit{transition density}, i.e., a family of functions $p_{s,t} : \Reals^d \times \Reals^d \to \Reals_+$ for all $0 \le s < t \le 1$, such that, for all points $x,y \in \Reals^d$ and all Borel sets $A \subset \Reals^d$,
\begin{align}\label{eq:transition_density}
	\PP[X_t \in A| X_s = x] = \int_A p_{s,t}(x,y) \dif y
\end{align}
(see, e.g., \citet[Chap.~V]{protter2005SDE}).

Consider the following stochastic control problem: Let $\cU$ be the set of \textit{controls}, i.e., measurable functions $u : \Reals^d \times [0,1] \to \Reals^d$. Any  $u \in \cU$ defines a diffusion process $X^u = \{X^u_t\}_{t \in [0,1]}$ by
\begin{align}\label{eq:controlled_diffusion}
	\dif X^u_t = \left(b(X^u_t,t) + u(X^u_t,t) \right)\dif t + \dif W_t, \qquad t \in [0,1];\, X^u_0 = x_0.
\end{align}
We say that $X^u$ is a diffusion controlled by $u$. Let a function $g : \Reals^d \to (0,\infty)$ be given. For each $u \in \cU$, we define the family of \textit{cost-to-go functions}
\begin{align}\label{eq:log_g_cost}
	J^u(x,t) \deq \E\Bigg[ \frac{1}{2}\int^1_t \|u_s\|^2\dif s -\log g(X^u_1)\Bigg|X^u_t = x\Bigg], \qquad x \in \Reals^d, t \in [0,1]
\end{align}
where $u_s$ is shorthand for $u(X^u_s,s)$. The \textit{value functions} $v : \Reals^d \times [0,1] \to \Reals_+$ are defined by
\begin{align}\label{eq:value_function}
	v(x,t) \deq \inf_{u \in \cU} J^u(x,t),
\end{align}
and we say that a control $u^* \in \cU$ is \textit{optimal} if $J^{u^*}(x,t) = v(x,t)$ for all $x$ and $t$.  The following theorem is, essentially, a synthesis of the results of \cite{pavon1989control} and \cite{daipra1991reciprocal}:

\begin{theorem}\label{thm:log_g_cost} Consider the control problem \eqref{eq:log_g_cost}. The value function $v$ is given by
	\begin{align}\label{eq:value_fctn}
		v(x,t) = -\log \E[g(X_1)|X_t = x],
	\end{align}
where the conditional expectation is with respect to the uncontrolled diffusion process \eqref{eq:generic_diffusion}. Moreover, the optimal control $u^*$ is given by $u^*(x,t) = -\nabla v(x,t)$, where the gradient is taken with respect to the space variable $x \in \Reals^d$, and the corresponding controlled diffusion $\{X^*_t\} = \{X^{u^*}_t\}$ has the transition density
\begin{align}\label{eq:optimal_transition_density}
p^*_{s,t}(x,y) = p_{s,t}(x,y)\exp\left(v(x,s)-v(y,t)\right),
\end{align}
where $p_{s,t}(\cdot)$ is the transition density \eqref{eq:transition_density} of the uncontrolled process.
\end{theorem}
This result, proved in Appendix~\ref{app:control}, also admits an information-theoretic interpretation. Let $\PP^0$ denote the probability law of the path $X_{[0,1]}$ of the uncontrolled diffusion process \eqref{eq:generic_diffusion} and let $\PP^u$ denote the corresponding object for the controlled diffusion \eqref{eq:controlled_diffusion}. Since $X$ and $X^u$ differ from each other by a change of drift, the probability measures $\PP^u$ and $\PP^0$ are mutually absolutely continuous, and the Radon--Nikodym derivative $\dif\PP^u/\dif\PP^0$ is given by the Girsanov formula \citep{protter2005SDE}
\begin{align}\label{eq:Girsanov}
	\frac{\dif\PP^u}{\dif\PP^0} = \exp\left(-\int^1_0 u_t^\trn \dif W_t + \frac{1}{2}\int^1_0 \|u_t\|^2\dif t\right),
\end{align}
where $u_t^\trn \dif W_t \deq \sum^d_{i=1} u_{i,t} \dif W_{i,t}$, with $u_{i,\cdot}$ and $\dif W_{i,\cdot}$ denoting the $i$th coordinates of $u$ and $W$ respectively. From \eqref{eq:Girsanov}, we can calculate the Kullback--Leibler divergence between $\PP^u$ and $\PP^0$:
\begin{align}\label{eq:KL_u0}
	D(\PP^u\|\PP^0) = \E_{\PP^u}\left[\log \frac{\dif\PP^u}{\dif\PP^0}\right] = \E\left[\frac{1}{2}\int^1_0 \|u_t\|^2 \dif t\right].
\end{align}
Therefore, by Theorem~\ref{thm:log_g_cost}, for any control $u \in \cU$, we can write
\begin{align}\label{eq:KL_optimality}
	- \log \E[g(X_1)|X_0 = x] \le D(\PP^u \| \PP^0) - \E[\log g(X^u_1)|X^u_0 = x],
\end{align}
with equality if and only if $u = u^*$. An inequality of this form holds more generally for real-valued measurable functions of the entire path $X_{[0,1]}$ \citep{boue1998variational}.

We will now demonstrate how both the problem of sampling and the problem of variational inference can be addressed via the above theorem.

\subsection{Exact sampling: the F\"ollmer drift}
\label{ssec:follmer}

Recall that, in the context of exact sampling, the objective is to construct a diffusion process $\{X_t\}_{t \in [0,1]}$, such that $X_1$ has a given target distribution $\mu$. We will consider the case when $\mu$ is absolutely continuous with respect to the standard Gaussian measure $\gamma_d$ and let $f$ denote the Radon--Nikodym derivative $\dif \mu/\dif \gamma_d$. This problem goes back to a paper of \cite{schrodinger1931bridge}; for rigorous treatments, see, e.g., \cite{jamison1975markov}, \cite{follmer1985reversal}, \cite{daipra1991reciprocal}, \cite{lehec2013entropy}, \cite{eldan2018diffusion}. The derivation we give below is not new (see, e.g., \citet[Thm.~3.1]{daipra1991reciprocal}), but the route we take is somewhat different in that we make the stochastic control aspect more explicit. 

We take $b(x,t) \equiv 0$ and $X_0 = 0$ in \eqref{eq:generic_diffusion}. Then the diffusion process $\{X_t\}$ is simply the standard $d$-dimensional Brownian motion $\{W_t\}$, which has the Gaussian transition density
\begin{align*}
	p_{s,t}(x,y) = \frac{1}{(2\pi (t-s))^{d/2}} \exp\left(-\frac{1}{2(t-s)}\|x-y\|^2\right).
\end{align*}
Now consider the control problem \eqref{eq:log_g_cost} with $g = f$. By Theorem~\ref{thm:log_g_cost}, the value function $v$ is given by $v(x,t) = - \log \E[f(W_1)|W_t = x]$, and can be computed explicitly. For $0 \le t < 1$, we have
\begin{align*}
	e^{-v(x,t)} &= \E[f(W_1)|W_t = x] \\
	&= \frac{1}{(2\pi(1-t))^{d/2}} \int_{\Reals^d} f(y) \exp\left(-\frac{1}{2(1-t)}\|x-y\|^2\right)\dif y \\
	&= Q_{1-t}f(x),
\end{align*}
where $Q$ denotes the Euclidean heat semigroup \eqref{eq:heat}. Hence, $v(x,t) = -\log Q_{1-t}f(x)$, and the optimal diffusion process $\{X^*_t\}$ has the drift $u^*(x,t) = -\nabla v(x,t) = \nabla \log Q_{1-t}f(x)$. Following \cite{lehec2013entropy} and \cite{eldan2018diffusion}, we will refer to $u^*$ as the \textit{F\"ollmer drift} in the sequel.

It remains to show that $X^*_1 \sim \mu$. Using the formula \eqref{eq:optimal_transition_density} for the transition density of $X^*$ together with the fact that $e^{-v(y,1)} =  f(y)$ and $e^{-v(0,0)} = \E[f(W_1)] = \int f\dif\gamma_d = 1$, we see that $p^*_{0,1}(0,y)\dif y = f(y)\gamma_d(\dif y)$. Then, for any Borel set $A \subseteq \Reals^d$,
\begin{align*}
	\PP[X^*_1 \in A] = \int_A p^*_{0,1}(0,y) \dif y = \int_A f(y) \gamma_d(\dif y) = \mu(A).
\end{align*}
Moreover, using the entropy inequality \eqref{eq:KL_optimality}, we can show that the F\"ollmer drift is optimal in the following strong sense: Consider any control $u \in \cU$ with $X^u_0 = 0$ and with the property that ${\rm Law}(X^u_1) = \mu$. For any such control,
\begin{align*}
	\E[\log f(X^u_1)|X^u_0=0] = \int_{\Reals^d}\dif \mu\log f = \int_{\Reals^d}\dif\mu \log \frac{\dif\mu}{\dif\gamma_d} = D(\mu \| \gamma_d),
\end{align*}
while clearly $\log\E[f(W_1)]=0$. Therefore, it follows from \eqref{eq:KL_optimality} that, for any such control $u$,
\begin{align*}
	D(\PP^u\|\PP^0) = \frac{1}{2}\E\left[\int^1_0 \|u_t\|^2\dif t\right] \ge D(\mu \| \gamma_d),
\end{align*}
with equality if and only if $u = u^*$. Thus, the F\"ollmer drift has the minimal `energy' among all admissible controls that induce the distribution $\mu$ at $t=1$, and this energy is precisely the Kullback--Leibler divergence between $\mu$ and the standard Gaussian measure $\gamma_d$ \citep{daipra1991reciprocal,lehec2013entropy,eldan2018diffusion}.

\subsection{Variational inference}

We now turn to the problem of variational inference. We are given an $n$-tuple of observations $\bd{y} = (y_1,\ldots,y_n) \in \sY^n$, and wish to upper-bound the negative log-likelihood
\begin{align*}
	L_n(\bd{y};\theta) \deq \frac{1}{n} \sum^n_{i=1} L(y_i;\theta),
\end{align*}
where $L(y;\theta) \deq - \log \E[q(y|X_1)]$ and $\{X_t\}$ is the diffusion process \eqref{eq:generative_model}.

We  take $b = b(\cdot,\cdot;\theta)$ in \eqref{eq:generic_diffusion} and consider the control problem \eqref{eq:log_g_cost} with $g(x) = q(y|x)$ for some fixed $y \in \sY$. Then, by Theorem~\ref{thm:log_g_cost}, any control $u \in \cU$ gives rise to an upper bound on $L(y;\theta)$:
\begin{align*}
	L(y;\theta) \le  \E\Bigg[\frac{1}{2}\int^1_0 \|u_t\|^2 \dif t - \log q(y|X^u_1)\Bigg|X^u_0 = x\Bigg] =: F^u(y; \theta),
\end{align*}
where the quantity on the right-hand side can be thought of as the \textit{variational free energy} that depends on the choice of the control $u$, and equality is achieved when $u = u^*$. While the structure of the optimal control $u^*$ is described in Theorem~\ref{thm:log_g_cost}, it may not be possible to derive it in closed form. However, we can fix a class $\tilde{\cU} \subset \cU$ of tractable suboptimal controls and upper-bound $L(y;\theta)$ by $\inf_{u \in \tilde{\cU}} F^u(y;\theta)$. For example, we can take $\tilde{\cU}$ to consist of all controls of the form $u(x,t) = \phi - b(x,t;\theta)$ for some $\phi \in \Reals^d$. In that case, $X^u_t$ is the sum of the Brownian motion $W_t$ and the affine drift $x + t \phi$, and consequently
\begin{align*}
	F^u(y;\theta) = \E\Bigg[\frac{1}{2}\int^1_0 \|\phi - b(x + t\phi + W_t; \theta)\|^2 \dif t - \log q(y|x + \phi + W_1)\Bigg],
\end{align*}
where the expectation is taken with respect to the standard Brownian motion $W$. Another possiblity is to consider controls of the form $u(x,t) = A x - b(x,t;\theta)$, for some $A \in \Reals^{d \times d}$. The corresponding controlled diffusion is the Ornstein--Uhlenbeck process $X^u_t = e^{At}x + \int^t_0 e^{A(t-s)}\dif W_s$, and the variational free energy can be minimized over $A \in \Reals^{d \times d}$.

\section{Expressiveness}
\label{sec:expressive}

Now that we have shown that generative models of the form \eqref{eq:generative_model} allow for both sampling and variational inference, we turn to the analysis of their \textit{expressiveness}. Specifically, our objective is to show that, by working with a suitable structured class of drifts $b(\cdot,\cdot;\theta)$, we can achieve approximate sampling from a rich class of distributions at the terminal time $t=1$.

Let $\mu$ be the target probability measure for $X_1$. We assume that $\mu$ is absolutely continuous with respect to $\gamma_d$ and let $f$ denote the Radon--Nikodym derivative $\dif\mu/\dif\gamma_d$. From Section~\ref{ssec:follmer} we know that the diffusion process governed by the It\^o SDE
\begin{align}
	\dif X_t = b(X_t,t) \dif t + \dif W_t, \qquad X_0 = 0
\end{align}
with the F\"ollmer drift $b(x,t) = \nabla \log Q_{1-t}f(x)$ has the property that $\mu = {\rm Law}(X_1)$, and, moreover, it is optimal in the sense that it minimizes the `energy' $\frac{1}{2}\int^1_0 \E\|u_t\|^2 \dif t$ among all adapted drifts $\{u_t\}$ that result in distribution $\mu$ at time $t=1$. The main result of this section is as follows: If the Radon--Nikodym derivative $f$ can be approximated efficiently by multilayer feedforward neural nets, then, for any $\eps > 0$, there exists a drift $\wh{b}(x,t) = \wh{b}(x,t;\theta)$ that can be implemented \textit{exactly} by a neural net whose parameters $\theta$ do not depend on time or space, and the terminal law $\wh{\mu} \deq {\rm Law}(\wh{X}_1)$ of the diffusion process
\begin{align}
	\dif\wh{X}_t = \wh{b}(\wh{X}_t,t) \dif t + \dif W_t, \qquad \wh{X}_0 = 0
\end{align}
is an $\eps$-approximation to $\mu$ in the KL-divergence: $D(\mu\|\wh{\mu}) \le \eps$. Moreover, the size of the neural net that implements the approximate F\"ollmer drift $\wh{b}$ can be estimated explicitly in terms of the size of a suitable approximating neural net for $f$.

We begin by imposing some assumptions on $f$. The first assumption is needed to guarantee enough regularity for the F\"ollmer drift:

\begin{assumption}\label{as:f} The function $f$ is differentiable, both $f$ and $\nabla f$ are $L$-Lipschitz, and there exists a constant $c \in (0,1]$, such that $f \ge c$ everywhere.
\end{assumption}
\sloppypar\noindent This assumption is satisfied, for example, by Gibbs measures of the form $\mu(\dif x) = Z^{-1}e^{-\frac{1}{2}\|x\|^2-F(x)}\dif x$ with a differentiable potential $F : \Reals^d \to \Reals_+$, such that both $F$ and $\nabla F$ are Lipschitz, and $F$ is bounded from above; see Appendix~\ref{app:f_regularity} for details.

Next, we introduce the assumptions pertaining to the approximability of $f$ by neural nets. Let $\sigma : \Reals \to \Reals$ be a fixed nonlinearity. Given a vector $w \in \Reals^n$ and scalars $\alpha,\beta$, define the function
\begin{align*}
	N^\sigma_{w,\alpha,\beta} : \Reals^n \to \Reals, \qquad N^\sigma_{w,\alpha,\beta}(x) \deq \alpha\cdot\sigma\left(w^\trn x+\beta\right).
\end{align*}
For $\ell \ge 2$, we define the class $\cN^\sigma_\ell$ of $\ell$-layer feedforward neural nets  with activation function $\sigma$ recursively as follows: $\cN^\sigma_2$ consists of all functions of the form $x \mapsto \sum^m_{i=1}N^\sigma_{w_i,\alpha_i,\beta_i}(x)$ for all $m \in \Naturals$, $w_1,\ldots,w_m \in \Reals^d$, $\alpha_1,\ldots,\alpha_m,\beta_1,\ldots,\beta_m \in \Reals$, and, for each $\ell \ge 2$,
\begin{align*}
	\cN^\sigma_{\ell+1} &\deq \bigcup_{k \ge 1}\bigcup_{m \ge 1}\Bigg\{ x \mapsto \sum^m_{i=1} N^\sigma_{w_i,\alpha_i,\beta_i}(h_1(x),\ldots,h_k(x)) :\nonumber\\
	&\qquad \alpha_1,\ldots,\alpha_m,\beta_1,\ldots,\beta_m \in \Reals, w_1,\ldots,w_m \in \Reals^k, h_1,\ldots,h_k \in \cN^\sigma_{\ell}\Bigg\}.
\end{align*}
Thus, each element of $\cN^\sigma_\ell$ is a function that represents computation by a directed acyclic graph, where each node receives inputs $u_1,\ldots,u_k$, performs a computation of the form $(u_1,\ldots,u_k) \mapsto \sigma(w_1 u_1 + \ldots + w_k u_k + \beta)$, and communicates the outcome of the computation to all the nodes in the next layer. We refer to $\ell$ as the \textit{depth} of the neural net, and define the \textit{size} of the neural net as the total number of nodes in its computation graph. We will denote by $\cN^\sigma_{\ell,s}$ the collection of all neural nets with depth $\ell$ and size $s$. All these definitions extend straightforwardly to the case of neural nets with vector-valued output and to the case where each node may have a different activation function.

We assume that the activation function $\sigma$ is differentiable and \textit{universal}, in the sense that any univariate Lipschitz function which is nonconstant on a bounded interval can be approximated arbitrarily well by an element of $\cN^\sigma_2$:

\begin{assumption}\label{as:universal_activation} The activation function $\sigma : \Reals \to \Reals$ is differentiable. Moreover, there exists a constant $c_\sigma > 0$ depending only on $\sigma$, such that the following holds: For any $L$-Lipschitz function $h : \Reals \to \Reals$ which is constant outside the interval $[-R,R]$ and for any $\delta > 0$, there exist real numbers $a, \{(\alpha_i,\beta_i,\gamma_i)\}^m_{i=1}$, where $m \le c_\sigma \frac{RL}{\delta}$, such that the function
	\begin{align}\label{eq:2layer}
		\tilde{h}(x) = a + \sum^m_{i=1} \alpha_i \sigma(\beta_i x + \gamma_i)
	\end{align}
	satisfies $\sup_{x \in \Reals}|\tilde{h}(x)-h(x)| \le \delta$.
\end{assumption}
\begin{remark} {\em Apart from differentiability, this is the same assumption made by \cite{eldan2016depth}. For example, it holds for differentiable sigmoidal activation functions, i.e., monotonic functions that satisfy $\lim_{u \to -\infty}\sigma(u) = a$ and $\lim_{u \to +\infty}\sigma(u) = b$ for some $a \neq b$. The popular rectified linear unit (or ReLU) activation function $u \mapsto u \vee 0$ is universal in the above sense but not differentiable. However, we can replace it by the differentiable softplus function $u \mapsto \log(1+e^{cu})$, where increasing the value of $c > 0$ results in finer approximations to the ReLU. Also, note that the function $\tilde{h}$ differs from the elements of $\cN^\sigma_2$ by the presence of the constant term $a$. However, the constant function $x \mapsto a$ can be implemented by $N^\sigma_{0,a/\sigma(z),z}$, for any $z \in \Reals$ such that $\sigma(z) \neq 0$. Thus, we will refer to functions of the form \eqref{eq:2layer} as $2$-layer neural networks of size $m+1$.}
\end{remark}
\noindent  We also make the following assumption regarding approximability of $f$ by neural nets:

\begin{assumption}\label{as:f_nn_apx} For any $R > 0$ and $\eps > 0$, there exists a neural net $\wh{f} \in \cN^\sigma_{\ell,s}$ with $\ell,s \le {\rm poly}(1/\eps,d,L,R)$, such that
	\begin{align}
		\sup_{x \in \sB^d(R)} |f(x)-\wh{f}(x)| \le \eps \quad \text{and} \quad \sup_{x \in \sB^d(R)} \|\nabla f(x) - \nabla \wh{f} (x) \| \le \eps.
	\end{align}
\end{assumption}
\begin{remark}{\em Typical results on neural net approximation are concerned with approximating a given function uniformly on a given compact set. By contrast, Assumption~\ref{as:f_nn_apx} requires uniform approximability of both $f$ and its gradient $\nabla f$ on a compact set by some neural net $\wh{f}$ and its gradient $\nabla \wh{f}$. Such simultaneous approximation guarantees can also be found in the literature, see, e.g., \cite{hornik1990nn_apx,yukich95nnapx,li1996nn_apx}. See \citet{safran2017depth} for a discussion of various trade-offs between depth and width (maximum number of neurons per layer) in neural net approximation.}
\end{remark}

We are now in a position to state the main result of this section:

\begin{theorem}\label{thm:expressive} Suppose Assumptions~\ref{as:f}--\ref{as:f_nn_apx} are in force. Let $L$ denote the maximum of the Lipschitz constants of $f$ and $\nabla f$. Then, for any $0 < \eps < 16L^2/c^2$, there exists a neural net $\wh{v} : \Reals^d \times [0,1] \to \Reals^d$ with size polynomial in $1/\eps,d,L,c,1/c$, such that the activation function of each neuron is an element of the set $\{\sigma,\sigma',{\rm ReLU}\}$, and the following holds: If $\{\wh{X}_t\}_{t \in [0,1]}$ is the diffusion process governed by the It\^o SDE
	\begin{align}
		\dif\wh{X}_t &= \wh{b}(\wh{X}_t,t)\dif t + \dif W_t, \qquad \wh{X}_0 = 0
	\end{align}
with the drift $\wh{b}(x,t) = \wh{v}(x,\sqrt{1-t})$, then $\wh{\mu} \deq {\rm Law}(\wh{X}_1)$ satisfies $D(\mu \| \wh{\mu}) \le \eps$.
\end{theorem}

\subsection{The proof of Theorem~\ref{thm:expressive}}

The proof relies on three key steps: First, we show that the heat semigroup $Q_t f(x)$ can be approximated by a finite sum of the form $\frac{1}{N}\sum_{n \le N}f(x+\sqrt{t}z_n)$ uniformly for all $x \in \sB^d(R)$ and all $t \in [0,1]$, where $z_1,\ldots,z_N \in \Reals^d$ lie in a ball of radius $\cO(\sqrt{d \log N})$.  This result is stated in Appendix~\ref{app:SAA_proof} and proved using empirical process methods. Next, replacing $f$ with a suitable neural net approximation $\wh{f}$, we build on this result to show that the F\"ollmer drift $\nabla \log Q_{1-t}f(x)$ can be approximated by a neural net using $\sigma$, $\sigma'$, and {\rm ReLU} as activation functions. This is the content of Theorem~\ref{thm:neural_net_apx} below (the proof is given in Appendix~\ref{app:neural_net_apx_proof}).  The third step uses Girsanov theory to upper-bound the approximation error that results from replacing the F\"ollmer drift by this neural net.

\begin{theorem}\label{thm:neural_net_apx} Let $0 < \eps < 4L/c$ and $R > 0$ be given. Then there exists a neural net $\wh{v} : \Reals^d \times [0,1] \to \Reals^d$ of size polynomial in $1/\eps,d,L,R,c,1/c$, such that the activation function of each neuron is an element of the set $\{\sigma,\sigma',{\rm ReLU}\}$, and the following holds:
\begin{align*}
	\sup_{x \in \sB^d(R)} \sup_{t \in [0,1]} \left\| \wh{v}(x,\sqrt{t}) - \nabla \log Q_{t}f(x)\right\| \le \eps
\end{align*}
and
\begin{align*}
	\max_{i \in [d]}\sup_{x \in \Reals^d} \sup_{t \in [0,1]} | \wh{v}_i(x,\sqrt{t}) | \le \frac{2L}{c}.
\end{align*}
\end{theorem}

We now complete the proof of Theorem~\ref{thm:expressive}. For any $R > 0$, Theorem~\ref{thm:neural_net_apx} guarantees the existence of a neural net $\wh{v} : \Reals^d \times [0,1] \to \Reals^d$ that satisfies
	\begin{align}\label{eq:apx_in_R_ball}
		\sup_{x \in \sB^d(R)}\sup_{t\in[0,1]} \left\| \wh{v}(x,\sqrt{t}) - \nabla \log Q_tf(x)\right\| \le \sqrt{\eps}
	\end{align}
	and
	\begin{align}\label{eq:global_bound}
		\max_{i \in [d]}\sup_{x \in \Reals^d} \sup_{t \in [0,1]} |\wh{v}_i(x,\sqrt{t})| \le \frac{2L}{c}.
	\end{align}
Let $\bd{\mu} \deq {\rm Law}(X_{[0,1]})$ and $\wh{\bd{\mu}} \deq {\rm Law}(\wh{X}_{[0,1]})$. The Girsanov formula gives
\begin{align*}
	D(\bd{\mu}\|\wh{\bd{\mu}}) &= \frac{1}{2}\int^1_0 \E\|b(X_t,t)-\wh{b}(X_t,t)\|^2 \dif t,
\end{align*}
where the interchange of the integral and the expectation follows from Fubini's theorem because both $b$ and $\wh{b}$ are bounded by Lemma~\ref{lm:Follmer_drift_regularity} in Appendix~\ref{app:f_regularity} and \eqref{eq:global_bound}. We now proceed to estimate the integrand. For each $t \in [0,1]$,
\begin{align*}
	&\E\|b(X_t,t)-\wh{b}(X_t,t)\|^2 \nonumber\\
	&=\E\left[\|b(X_t,t)-\wh{b}(X_t,t)\|^2\cdot\1{\{X_t \in \sB^d(R)\}}\right] + \E\left[\|b(X_t,t)-\wh{b}(X_t,t)\|^2\cdot\1{\{X_t \not\in \sB^d(R)\}}\right] \nonumber \\
	&=: T_1 + T_2,
\end{align*}
where $T_1 \le \eps$ by \eqref{eq:apx_in_R_ball}. To estimate $T_2$, we first observe that, since the F\"ollmer drift is bounded in norm by $L/c$ by Lemma~\ref{lm:Follmer_drift_regularity}, we have
\begin{align*}
	\PP\left\{ \sup_{t \in [0,1]} \|X_t\| \ge R \right\} \le \frac{\sqrt{d}+L/c}{R}
\end{align*}
\citep[Lemma~3.8]{bubeck2018LMC}. Therefore,
\begin{align*}
	T_2 &\le \frac{9dL^2}{c^2} \cdot \frac{\sqrt{d}+L/c}{R}.
\end{align*}
Choosing $R$ large enough to guarantee $T_2 \le \eps$ and putting everything together, we obtain $D(\bd{\mu}\|\wh{\bd{\mu}}) \le \eps$. Therefore, $D(\mu\|\wh{\mu}) \le D(\bd{\mu}\|\wh{\bd{\mu}}) \le \eps$ by the data processing inequality.

\section{Unbiased simulation}
\label{sec:simulation}

Now that we have shown that generative models with latent diffusions are capable of expressing a rich class of probability distributions, we turn to the problem of unbiased simulation. Specifically, given a function $g : \Reals^d \to \Reals$, we wish to estimate the expectation $\E[g(X_1)|X_0 = x]$, where $X = \{X_t\}_{t \in [0,1]}$ with $X_0 = x$ is a diffusion process of the form \eqref{eq:generative_model}. The simplest approach is to use the Euler--Maruyama scheme: Fix a partition $0 = t_0 < t_1 < \ldots < t_n < t_{n+1} = 1$ of $[0,1]$ and define the It\^o process $\{\tilde{X}_t\}_{t \in [0,1]}$ by $\tilde{X}_0 = x$ and
\begin{align}\label{eq:Euler}
	\tilde{X}_t = \tilde{X}_{t_i} + \int^t_{t_i} b(\tilde{X}_{t_i},t_i; \theta) \dif s + \int^t_{t_i} \dif W_s, \qquad t \in (t_i,t_{i+1}], i = 0,\ldots,n.
\end{align}
In particular, for each $1 \le i \le n+1$,
\begin{align*}
	\tilde{X}_{t_i} = \tilde{X}_{t_{i-1}} + b(\tilde{X}_{t_{i-1}},t_{i-1};\theta)(t_i - t_{i-1}) + W_{t_i} - W_{t_{i-1}}.
\end{align*}
We can then estimate the expectation $\E[g(X_1)]$ by $g(\tilde{X}_{t_{n+1}}) \equiv g(\tilde{X}_1)$, but this estimate is biased: if $g$ is, say, bounded, then
\begin{align*}
	|\E[g(X_1)]-\E[g(\tilde{X}_1)]| \le C_g(x) \cdot \max_{0 \le i \le n} (t_{i+1}-t_i),
\end{align*}
where $C_g(x) > 0$ is some constant that depends on $g$ and on the starting point $x$ \citep{graham2013simulation}. Recently, several authors \citep{bally2015parametrix,andersson2017unbiased,henry2017unbiased} have studied unbiased simulation of SDEs using Euler--Maruyama schemes with \textit{random} partitions, where the partition breakpoints are generated by a Poisson point process on the real line. In this section, we build on this line of work and present a scheme for unbiased simulation in the context of generative models of the form \eqref{eq:generative_model} that uses random partitions generated by arbitrary renewal processes \citep[Chap.~9]{kallenberg2002prob_book} with sufficiently well-behaved densities of interrenewal times. Our analysis closely follows that of \citet{henry2017unbiased}, but we provide a more refined analysis of the variance of the resulting estimators.

We first describe the simulation procedure. In what follows, we will drop the index $\theta$ from the drift to keep the notation clean. Let $\tau_1,\tau_2,\ldots$ be i.i.d.\ nonnegative random variables with an absolutely continuous distribution whose support contains the interval $[0,1+\eps]$ for some $\eps > 0$. Let $F_\tau$ and $f_\tau$ denote the cdf and the pdf of $\tau_1$. Let $T_0 = 0$ and 
\begin{align*}
	T_k \deq \left(\sum^k_{i=1} \tau_i\right) \wedge 1,\, k \ge 1  \qquad \text{and} \qquad N \deq \max \{ k : T_k < 1  \}.
\end{align*}
Define a process $\wh{X} = \{ \wh{X}_t \}_{t \in [0,1]}$ with $\wh{X}_0 = x$ as the Euler--Maruyama scheme \eqref{eq:Euler} on the random partition $0 = T_0 < T_1 < \ldots < T_N < T_{N+1} \equiv 1$ of $[0,1]$, and let
\begin{align}\label{eq:psi_def}
	\wh{\psi} \deq \frac{1}{1-F_\tau(1-T_N)} \cdot \left(g(\wh{X}_1)-g(\wh{X}_{T_N})\1_{\{N > 0\}}\right)   \cdot \prod^N_{k=1} \frac{1}{f_\tau(T_{k}-T_{k-1})}\wh{\cW}_k, 
\end{align}
where
\begin{align*}
	\wh{\cW}_k \deq \frac{\big(b(\wh{X}_{T_k},T_k)-b(\wh{X}_{T_{k-1}},T_{k-1})\big)^\trn\big(W_{T_{k+1}}-W_{T_k}\big)}{T_{k+1}-T_k}.
\end{align*}
This process can be interpreted as a deep generative model in the sense of \citet{rezende2014stochbackprop}, but with a \textit{random} number of layers. Specifically, let $\xi_1,\xi_2,\ldots \stackrel{{\rm i.i.d.}}{\sim} \gamma_d$ be independent of $\{\tau_i\}$, and define $\wh{X}^{(0)},\wh{X}^{(1)},\ldots,\wh{X}^{(N+1)}$ recursively by taking $\wh{X}^{(0)}=x$ and
\begin{align*}
	\wh{X}^{(k+1)} = \wh{X}^{(k)} + b(\wh{X}^{(k)},T_k) \cdot (T_{k+1}-T_k) +  (T_{k+1}-T_k)^{1/2}\xi_{k+1}, \qquad k = 0,1,\ldots,N.
\end{align*}
Then
\begin{align*}
	\wh{\psi} \stackrel{{\rm d}}{=} g(\wh{X}^{(N+1)}) \cdot \frac{1}{1-F_\tau(1-T_{N})} \cdot \prod^N_{k=1} \frac{\left(b(\wh{X}^{(k)},T_k)-b(\wh{X}^{(k-1)},T_{k-1})\right)^\trn\xi_{k+1}}{f_\tau(T_k-T_{k-1})\cdot(T_{k+1}-T_k)^{1/2}},
\end{align*}
where $\stackrel{{\rm d}}{=}$ denotes equality of probability distributions. We are now ready to state our main result on unbiased simulation (see Appendix~\ref{app:unbiased} for the proof):
\begin{theorem}\label{thm:unbiased} Suppose that the drift $b(x,t)$ is uniformly bounded, Lipschitz in $x$, and $\frac{1}{2}$-H\"older in $t$, i.e., for some constants $b_\infty > 0$ and  $L_b > 0$,
	\begin{align}\label{eq:smooth_drift}
		\|b(x,t)\| \le b_\infty \qquad\text{and}\qquad \|b(x,s)-b(y,t)\| \le L_b\left(\|x-y\| + |s-t|^{1/2}\right).
	\end{align}
for all $x,y \in \Reals^d$ and all $s,t \in [0,1]$. Suppose also that 
	\begin{align}\label{eq:tau_pdf}
		\frac{1}{f_\tau(s)} \le Ce^{as}, \qquad s \in(0,1)
	\end{align}
for some constants $C > 0$ and $a \ge 0$. Then, for any Lipschitz-continuous $g : \Reals^d \to \Reals$ with Lipschitz constant $L_g$, $\wh{\psi}$ is an unbiased estimator of $\E[g(X_1)|X_0 = x]$ with
	\begin{align}
		{\rm Var}[\wh{\psi}] \le \left(\frac{e^a}{1-F_\tau(1)}\right)^2 KM_N(\kappa),
	\end{align}
where $K = {\rm poly}(|g(x)|,L_b,L_g,b_\infty,d)$, $\kappa = \log {\rm poly}(C,L_b,L_g,b_\infty,d)$, and $M_N(\theta) \deq \E[\exp(\theta N)]$ is the moment-generating function of $N$.
\end{theorem}

\noindent For example, the type of drift used in the construction of Section~\ref{sec:expressive} has the property \eqref{eq:smooth_drift}. The key implication of Theorem~\ref{thm:unbiased} is that the variance of the estimator $\wh{\psi}$ is controlled by the moment-generating function of $N$, and is therefore related to the tail behavior of the sums $S_k \deq \sum^k_{i=1}\tau_i$. In some cases, one can calculate $M_N$ in closed form. For instance, if we take $\tau_1,\tau_2,\ldots  \stackrel{{\rm i.i.d.}}{\sim} {\rm Exp}(\lambda)$ for some $\lambda > 0$, then the estimator \eqref{eq:psi_def} reduces to the one introduced by \citet{henry2017unbiased}. Since $F_\tau(s) = 1-e^{-\lambda s}$ and $f_\tau(s) = \lambda e^{-\lambda s}$ for $s \ge 0$,  \eqref{eq:tau_pdf} holds with $C = 1/\lambda$ and $a =\lambda$; moreover, $N \sim {\rm Pois}(\lambda)$ with
$$
M_N(\theta) = \exp\big(\lambda (e^\theta - 1)\big).
$$
Thus, ${\rm Var}[\wh{\psi}]$ grows like $\exp(d^2)$, as already observed by \cite{henry2017unbiased}. One way to reduce the variance is to choose the $\tau_i$'s with lighter tails. To see this, we need estimates of $\Lambda_N$; the following lemma provides a computable upper bound:

\begin{lemma}\label{lm:N_mgf} Let $M_\tau$ denote the moment-generating function of $\tau$. Then
	\begin{align}\label{eq:N_mgf}
	M_N(\theta) \le 1 + e^\theta\inf_{\beta > 0} \Bigg\{ (\beta+1) e^{\theta\beta} + \sum^\infty_{k=0}\big(e^{\theta+1}M_\tau(-\beta)\big)^k\Bigg\}.
	\end{align}
\end{lemma}
As an example, suppose $\tau_1,\tau_2,\ldots$ are i.i.d.\ samples from the uniform distribution on $[0,T]$ for some $T > 1$. Then
$$
M_\tau(-\beta) = \frac{1}{\beta T}(1-e^{-\beta T}),
$$
and it is a matter of straightforward but lengthy algebra to show that $M_\tau(-\beta) \le e^{-2(\theta+1)}$ for all $\beta$ satisfying
\begin{align*}
	e^{-\beta T} \ge 2\left(1+e^{-2(\theta+1)}\log 2 - (2\theta+3)e^{-2(\theta+1)}\right).
\end{align*}
Using this in \eqref{eq:N_mgf} yields the estimate $M_N(\theta) \lesssim e^{{\rm poly}(\theta)}$. The density of a ${\rm Uniform}(0,T)$ random variable clearly satisfies \eqref{eq:tau_pdf}. Thus, applying Theorem~\ref{thm:unbiased} to the estimator \eqref{eq:psi_def} with $\tau_i \stackrel{{\rm i.i.d.}}{\sim} {\rm Uniform}(0,T)$, we see that its variance scales \textit{quasipolynomially} in $d$, i.e., ${\rm Var}[\wh{\psi}] \lesssim e^{{\rm polylog}(d)}$. However, choosing $\tau_i$'s with lighter tails will generally lead to larger values of $N$, i.e., a deeper generative model will be needed.

\newpage

\appendix

\section{The proof of Theorem~\ref{thm:log_g_cost}}
\label{app:control}

We first need some background on controlled diffusion processes, see, e.g., \cite{fleming1975control}. As in Section~\ref{sec:control}, let $\cU$ be the set of controls, where each $u \in \cU$ defines a controlled diffusion governed by the It\^o SDE
\begin{align*}
	\dif X^u_t = \left(b(X^u_t,t) + u(X^u_t,t) \right)\dif t + \dif W_t, \qquad t \in [0,1];\, X^u_0 = x_0.
\end{align*}
Let $c : \Reals^d \times \Reals^d \to \Reals_+$ and $\tilde{c} : \Reals^d \to \Reals_+$ be given. For each $u \in \cU$, we define the cost-to-go functions
\begin{align}\label{eq:cost_to_go}
	J^u(x,t) \deq \E\Bigg[ \int^1_t c(X^u_s,u_s)\dif s + \tilde{c}(X^u_1)\Bigg|X^u_t = x\Bigg], \qquad x \in \Reals^d, t \in [0,1]
\end{align}
where $u_s$ is shorthand for $u(X^u_s,s)$. The value functions $v : \Reals^d \times [0,1] \to \Reals_+$ are defined in \eqref{eq:value_function}. In general, finding an optimal control is difficult. However, a sufficient condition for optimality is given by the so-called \textit{verification theorem} from the theory of controlled diffusions (see, e.g., Chap.~VI of \citet{fleming1975control}): Suppose that there exists a function $v \in C^{2,1}(\Reals^d \times [0,1])$ that solves the Cauchy problem
\begin{align}\label{eq:Bellman}
	\frac{\partial v(x,t)}{\partial t} + \cL_t v(x,t) = - \min_{\alpha \in \Reals^d} \left\{ \alpha^\trn \nabla v(x,t) + c(x,\alpha)\right\} \text{ on } \Reals^d \times [0,1]; \qquad g(\cdot,1) = \tilde{c}(\cdot)
\end{align}
where $\cL_t$ is the (time-varying) generator of the diffusion \eqref{eq:generic_diffusion}:
\begin{align}\label{eq:generator}
	\cL_t h(x,t) \deq b(x,t)^\trn \nabla h(x,t) + \frac{1}{2} \tr \nabla^2 h(x,t)
\end{align}
for any $h \in C^{2,1}(\Reals^d \times [0,1])$, and where the gradient and the Hessian are taken with respect to the `space variable' $x \in \Reals^d$. Then $v$ is the value function for \eqref{eq:cost_to_go}, and the optimal control $u^*$ is given by
\begin{align}\label{eq:optimal_control}
	u^*(x,t) = \argmin_{\alpha \in \Reals^d} \left\{ \alpha^\trn \nabla v(x,t) + c(x,\alpha)\right\}.
\end{align}
The PDE \eqref{eq:Bellman} is called the \textit{Bellman equation} associated to the control problem \eqref{eq:cost_to_go}.

\begin{remark} {\em In fact, the control \eqref{eq:optimal_control} is optimal among a much wider class of \textit{adapted controls}, i.e., all stochastic processes $\{u_t\}_{t \in [0,1]}$ adapted to the filtration $\{\cF_t\}$. The class $\cU$ defined above consists of so-called \textit{Markov controls}, where $u_t$ is a deterministic function of $X^u_t$ and $t$. In that case, the controlled diffusion $X^u$ is a Markov process.}
\end{remark}

We now turn to the proof of Theorem~\ref{thm:log_g_cost}. The first step is to use the logarithmic transformation due to \cite{fleming1978exit}; see also \cite{fleming1985fundamental,sheu1991transition}. Consider the function $h(x,t) \deq \E[g(X_1)|X_t = x]$. By the Feynman--Kac formula \citep[Thm.~24.1]{kallenberg2002prob_book}, this function is a $C^{2,1}$ solution of the Cauchy problem
\begin{align}\label{eq:FK}
	\frac{\partial h}{\partial t} + \cL_t h = 0 \text{ on } \Reals^d \times [0,1]; \qquad h(\cdot,1) = g(\cdot).
\end{align}
It is a matter of simple calculus to verify that $v(x,t) = -\log h(x,t)$ solves the Cauchy problem
\begin{align}\label{eq:HJB}
	\frac{\partial v}{\partial t} + \cL_t v = \frac{1}{2}\|\nabla v\|^2 \text{ on } \Reals^d \times [0,1]; \qquad v(\cdot,1) = - \log g(\cdot).
\end{align}
Moreover, using the variational representation
\begin{align*}
	\frac{1}{2}\|\nabla v\|^2 = -\min_{\alpha \in \Reals^d} \left\{\alpha^\trn \nabla v + \frac{1}{2}\|\alpha\|^2\right\},
	\end{align*}
	where the optimizer is given by $\alpha^* = - \nabla v$, it is readily verified that \eqref{eq:HJB} is the Bellman equation \eqref{eq:Bellman} associated to the control problem \eqref{eq:log_g_cost}. Hence, by the verification theorem, $v(x,t) = -\log h(x,t)$ is the value function we seek, and the optimal control is given by $u^*(x,t) = -\nabla v(x,t)$. 
	
	Now consider the diffusion process
	\begin{align*}
		\dif X^*_t = \big(b(X^*_t,t)+\nabla \log h(X^*_t,t)\big)\dif t + \dif W_t,
	\end{align*}
	which satisfies
	\begin{align*}
		-\log\E[g(X_1)|X_t = x] &= \E\Bigg[\frac{1}{2}\int^1_T \|\nabla \log h(X^*_s,s)\|^2 \dif s - \log g(X^*_1)\Bigg|X^*_t = x \Bigg] \\
		&= \min_{u \in \cU} \E\Bigg[\frac{1}{2}\int^1_t \|u_s\|^2\dif s - \log g(X^u_1)\Bigg|X^u_t = x\Bigg].
	\end{align*}
Since $h$ solves \eqref{eq:FK}, the transition density of $\{X^*_t\}$ is given by \eqref{eq:optimal_transition_density} by a result of \citet{jamison1975markov} and \citet{daipra1991reciprocal}.

\section{Regularity properties of $f$ and the F\"ollmer drift}
\label{app:f_regularity}

We first show that Assumption~\ref{as:f} holds for Gibbs measures
\begin{align*}
	\mu(\dif x) = Z^{-1}e^{-\frac{1}{2}\|x\|^2-F(x)}\dif x
\end{align*}
with sufficiently well-behaved potentials $F$. Suppose that $F : \Reals^d \to \Reals_+$ is differentiable, and both $F$ and $\nabla F$ are $L$-Lipschitz. Then $f = \dif\mu/\dif\gamma_d = {\rm const} \cdot e^{-F}$, and the Lipschitz continuity of $f$ follows from the Lipschitz continuity of $u \mapsto e^{-u}$ on $[0,\infty)$:
\begin{align*}
	|e^{-F(x)}-e^{-F(y)}| \le |F(x)-F(y)| \le L\|x-y\|.
\end{align*}
Likewise, the Lipschitz continuity of $\nabla f$ follows from the Lipschitz continuity of $\nabla F$: since $\nabla e^{-F} = -e^{-F}\nabla F$, we have
\begin{align*}
	\|\nabla e^{-F(x)} - \nabla e^{-F(y)} \| &\le e^{-F(x)} \|\nabla F(x)-\nabla F(y)\| + \|\nabla F(y)\| |e^{-F(x)}-e^{-F(y)}| \nonumber\\
	&\le (L+L^2)\|x-y\|.
\end{align*}
Finally, suppose that $F$ is also bounded from above, $F \le a$ for some $a > 0$. Then $f \ge c$ everywhere, where $0 < c \le 1$ because both $\mu$ and $\gamma_d$ are probability measures.

We will also need the following simple lemma:

\begin{lemma}[Regularity of the F\"ollmer drift]\label{lm:Follmer_drift_regularity} Under Assumption~\ref{as:f}, the F\"ollmer drift $b(x,t) = \nabla \log Q_{1-t}f(x)$ is bounded in norm by $L/c$ and is Lipschitz with Lipschitz constant $L/c + L^2/c^2$, where $L$ is the maximum of the Lipschitz constants of $f$ and $\nabla f$.
\end{lemma}
\begin{proof} The heat semigroup $Q_t f(x) = \E[f(x+\sqrt{t}Z)]$, $Z \sim \gamma_d$, commutes with the gradient operator: for any differentiable and Lipschitz $f : \Reals^d \to \Reals$, $\partial_{i}Q_tf = Q_t\partial_{i}f$ for all $i \in [d]$ \citep[Corollary~2.2.8]{stroock2008PDEs}. Therefore, since $f(x) \ge c$ and $\|\nabla f(x)\| \le L$ for all $x$, we have $Q_tf(x)\ge c$ and $\|\nabla Q_tf(x)\| \le L$ for all $x \in \Reals^d$ and all $t \ge 0$. Consequently, for any $x \in \Reals^d$ and $t \in [0,1]$
\begin{align*}
	\|b(x,t)\| &= \left\|\frac{\nabla Q_{1-t} f(x)}{Q_{1-t} f(x)}\right\| \le \frac{L}{c}.
\end{align*}
Also, since $\nabla f$ is Lipschitz, $\|\nabla Q_tf(x)-\nabla Q_tf(x')\| \le L\|x-x'\|$ for any $x,x' \in \Reals^d$ and $t \in [0,1]$, and thus
\begin{align*}
	\|b(x,t)-b(x',t)\| &= \left\|\frac{\nabla Q_{1-t}f(x)}{Q_{1-t}f(x)} - \frac{\nabla Q_{1-t}f(x')}{Q_{1-t}f(x')}\right\| \\
	&\le \frac{\|\nabla Q_{1-t}f(x)-\nabla Q_{1-t}f(x')\|}{Q_{1-t}f(x')} + \|b(x,t)\| \cdot \frac{|Q_{1-t}f(x)-Q_{1-t}f(x')|}{Q_{1-t}f(x')} \\
	&\le \left(\frac{L}{c}+\frac{L^2}{c^2}\right)\|x-x'\|,
\end{align*}
and the proof is complete.
\end{proof}

\section{Uniform approximation of the heat semigroup by a finite sum}
\label{app:SAA_proof}

In this appendix, we prove the following result, which is used in the proof of Theorem~\ref{thm:neural_net_apx}:

\begin{theorem}\label{thm:SAA} For any $\eps > 0$ and any $R > 0$, there exist $N = {\rm poly}(1/\eps,d,L,R)$ points $z_1,\ldots,z_N \in \Reals^d$, for which the following holds:
\begin{subequations}
\begin{align*}
 \max_{n \le N}\|z_n\| \le 8\sqrt{(d+6)\log N} &\\
 \sup_{x \in \sB^d(R)} \sup_{t \in [0,1]} \left|\frac{1}{N}\sum^N_{n=1}f(x+\sqrt{t}z_n) - Q_t f(x)\right| \le \eps & \\
\sup_{x \in \sB^d(R)} \sup_{t \in [0,1]} \left\|\frac{1}{N}\sum^N_{n=1}\nabla f(x+\sqrt{t}z_n) - \nabla Q_t f(x)\right\| \le \eps &
\end{align*}
\end{subequations}
\end{theorem}

We gather some preliminaries first. We recall the definition of the Orlicz exponential norm of order $2$ \citep[Sec.~2.3]{Gine_Nickl}: for a real-valued random variable $U$,
\begin{align*}
	\|U\|_{\psi_2} \deq \inf\left\{ c > 0 : \E \exp\left(\frac{|U|}{c}\right)^2 \le 2\right\}.
\end{align*}
The $\psi_2$ norm dominates the $L^2$ norm $\|U\|_2 \deq (\E|U|^2)^{1/2}$: $\|U\|_2 \le \|U\|_{\psi_2}$. A simple application of Markov's inequality leads to the following tail bound:
	\begin{align}\label{eq:psi2_tail_bound}
		\PP\left[|U| \ge t\|U\|_{\psi_2}\right] \le \frac{1}{e^{t^2}-1}.
	\end{align}

\begin{lemma}\label{lm:Gaussian_Orlicz_norm} Let $U = \|Z\|$, where $Z \sim \gamma_d$. Then $\|U\|_{\psi_2} \le \sqrt{d} + \sqrt{6}$.
\end{lemma}
\begin{proof} If $F : \Reals^d \to \Reals$ is $1$-Lipschitz, then the centered random variable $\xi = F(Z) - \E F(Z)$ has subgaussian tails \citep[Theorem~5.6]{Boucheron_etal_book}:
	\begin{align*}
		\PP\left\{|\xi| \ge t\right\} \le 2e^{-t^2/2} \text{ for all $t > 0$}.
	\end{align*}
This implies that $\|\xi\|_{\psi_2} \le \sqrt{6}$ \citep[Eq.~(2.25)]{Gine_Nickl}. Taking $F(Z)=U$ and using the triangle inequality, we obtain
\begin{align*}
	\|U\|_{\psi_2} &\le \E U + \|U-\E U\|_{\psi_2} \le \sqrt{d} + \sqrt{6},
\end{align*}
where $\E U \le \| U \|_2 = \sqrt{d}$ by Jensen's inequality.
\end{proof}
\noindent Let $U_1,\ldots,U_N$, $N \ge 2$, be a collection of (possibly dependent) random variables with finite $\psi_2$ norms. Then we have the following maximal inequality:
\begin{align}\label{eq:maxineq}
	\left\| \max_{j \le N}|U_j|\right\|_{\psi_2} \le 4\sqrt{\log N} \max_{j \le N} \|U_j\|_{\psi_2},
\end{align}
(Lemma~2.3.3 in \citet{Gine_Nickl}). 

We also need some results on suprema of empirical processes. Let $\cG$ be a class of real-valued functions on some measurable space $\sZ$. We say that a positive function $F : \sZ \to \Reals_+$ is an \textit{envelope} of $\cG$ if $|g(z)| \le F(z)$ for all $g \in \cG$ and $z \in \sZ$. Let $Z_1,\ldots,Z_N$ be i.i.d.\ random elements of $\sZ$ with probability law $P$ and denote by $P_N$ the corresponding empirical distribution, i.e., $P_N(A) = N^{-1}\sum_{n \le N} \1_{\{Z_n \in A\}}$ for all measurable sets $A \subset \sZ$. We will use the linear functional notation for expectations, i.e., $Pg \deq \E_P[g(Z)]$ and $P_Ng \deq \E_{P_N}[g(Z)] = N^{-1}\sum_{n \le N}g(Z_n)$. We are interested in the quantity
\begin{align*}
	\|P_N - P\|_\cG \deq \sup_{g \in \cG} |P_Ng - Pg|,
\end{align*}
which is a random variable under standard regularity assumptions on $\cG$, such as separability. The expected supremum $\E \|P_N - P\|_\cG$ is controlled by the covering numbers of $\cG$. The $L^2(Q)$ covering numbers of $\cG$ with respect to a probability measure $Q$ on $\sZ$ are defined by
\begin{align*}
	 N(\cG,L^2(Q),\eps) &\deq \min \Big\{ K \, : \, \text{ there exist $f_1,\ldots,f_K \in L^2(Q)$} \nonumber\\
	& \qquad \qquad \qquad \text{such that $\sup_{g \in \cG}\min_{k \le K} \|g-f_k\|_{L^2(P)} \le \eps$}\Big\}.
\end{align*}
The \textit{Koltchinskii--Pollard $\eps$-entropy} of $\cG$ is given by
\begin{align*}
	H(\cG,F,\eps) \deq \sup_{Q} \sqrt{\log 2N(\cG,L^2(Q),\eps \|F\|_{L^2(Q)})},
\end{align*}
where the supremum is over all probability measures $Q$ supported on finitely many points of $\sZ$. Then we have the following bound on the expectation of $\|P_N - P\|_\cG$ (Theorem~3.54 and Eq.~(3.177) in \citet{Gine_Nickl}):
\begin{lemma}\label{lm:expect_sup} Let $\cG$ be a class of functions containing $0$, such that
	\begin{align*}
		J(\cG,F) \deq \int^\infty_0 H(\cG,F,\eps) \dif\eps < \infty.
	\end{align*}
Let $Z_1,\ldots,Z_N$ be i.i.d.\ copies of a random element $Z$ of $\sZ$ with probability law $P$, such that $F \in L^2(P)$. Then
\begin{align*}
	\E\|P_N - P\|_\cG \le \frac{8\sqrt{2}J(\cG,F)\|F\|_{L^2(P)}}{\sqrt{N}}.
\end{align*}
\end{lemma}
\noindent We also have the following generalization of Talagrand's concentration inequality to unbounded classes of functions, due to \citet{adamczak2008suprema} (see also Sec.~2.3 in \citet{Koltchinskii_StFlour}):

\begin{lemma} Let $\cG$ be a class of real-valued functions on $\sZ$ with envelope $F$. Then there exists an absolute constant $C > 0$, such that, for any $\gamma > 0$,
	\begin{align*}
		\PP \left\{ \|P_N - P\|_\cG \ge C\left[\E\|P_N-P\|_\cG + \sigma_P(\cG)\sqrt{\frac{\gamma}{N}} + \left\|\max_{n \le N}F(Z_n)\right\|_{\psi_2}\frac{\sqrt{\gamma}}{N}\right]\right\} \le e^{-\gamma},
	\end{align*}
	where 
	\begin{align*}
		\sigma^2_P(\cG) \deq \sup_{g \in \cG}\left(Pg^2 - (Pg)^2\right).
	\end{align*}
\end{lemma}
With these preliminaries out of the way, we have the following result:

\begin{lemma}\label{lm:app_SAA} Let $g : \Reals^d \to \Reals$ be $L$-Lipschitz with respect to the Euclidean norm. Let $Z_1,\ldots,Z_N$ be i.i.d.\ copies of a $d$-dimensional random vector $Z$, such that $U \deq \|Z\|$ has finite $\psi_2$ norm.  Then there exists an absolute constant $C > 0$, such that, for any $\gamma > 0$,
\begin{align}
&	\sup_{x \in \sB^d(R)}\sup_{t \in [0,1]} \left|\frac{1}{N}\sum^N_{n=1}g(x+\sqrt{t}Z_i) - \E[g(x+\sqrt{t}Z)]\right|\nonumber\\
& \qquad \le C\left[ \frac{16L\sqrt{6\pi R d}((R \vee 1)+\|U\|_{\psi_2})}{\sqrt{N}} + 5L \left((R \vee 1) + \|U\|_{\psi_2}\right) \sqrt{\frac{\gamma}{N}} 
\right]\label{eq:gtail}
\end{align}
with probability at least $1-e^{-\gamma}$.
\end{lemma}

\begin{proof} For each $x \in \Reals^d$ and $t\ge 0$ let $g_{x,t}(z) \deq g(x+\sqrt{t}z)$. Let $P$ denote the probability law of $Z$. Since $P_N g_{x,t} - P g_{x,t} = P_N (g_{x,t}-g_{0,0}) - P(g_{x,t}-g_{0,0})$ for all $x,t$, where $g_{0,0}(\cdot) = g(0)$ is a constant, we can replace each $g_{x,t}$ with $\bar{g}_{x,t} \deq g_{x,t} - g_{0,0}$, introduce the function class $\cG \deq \{ \bar{g}_{x,t} : x \in \sB^d(R), t \in [0,1]\}$, and analyze the empirical process supremum
	\begin{align*}
		\| P_N - P \|_\cG = \sup_{x \in \sB^d(R)}\sup_{t\in[0,1]} |P_N \bar{g}_{x,t} - P \bar{g}_{x,t} |.
	\end{align*}
Define the function $F(z) \deq L((R \vee 1) + \|z\|)$. Since $\|\cdot\|_2 \le \|\cdot\|_{\psi_2}$, $F \in L^2(P)$. By Lipschitz continuity, for all $z \in \Reals^d$, $x \in \sB^d(R)$, $t \in [0,1]$, we have
\begin{align*}
	|\bar{g}_{x,t}(z)| \le |g(x+\sqrt{t}z)-g(0)| \le L\|x+\sqrt{t}z\| \le F(z),
\end{align*}
so $F$ is a square-integrable envelope of $\cG$. Moreover, for any probability measure $Q$ supported on finitely many points in $\Reals^d$ and for all $x,x' \in \sB^d(R)$ and $t,t' \in [0,1]$,
\begin{align*}
	\| \bar{g}_{x,t} - \bar{g}_{x',t'}\|_{L^2(Q)} \le \| F \|_{L^2(Q)} \cdot (\|x-x'\| + |t-t'|^{1/2}).
\end{align*}
Thus we can estimate the $L^2(Q)$ covering numbers of $\cG$ by
\begin{align*}
	N(\cG,L^2(Q),\eps\|F\|_{L^2(Q)}) \le N(\sB^d(R),\|\cdot\|,\eps/2) \cdot N([0,1],|\cdot|,\eps^2/4).
\end{align*}
Using standard volumetric estimates on the covering numbers of $\ell_2$ balls, we obtain the following bound on the Koltchinskii--Pollard entropy of $\cG$:
\begin{align*}
	H(\cG,F,\eps) \le \left( 4d \log \frac{2\sqrt{3R}}{\eps}\right)_+
\end{align*}
where $(u)_+ \deq u \vee 0$, and therefore
\begin{align*}
	J(\cG,F) = \int^\infty_0 H(\cG,F,\eps)\dif\eps \le 2\sqrt{3 \pi R d}.
\end{align*}
Lemma~\ref{lm:expect_sup} then gives
\begin{align}
	\E \|P_N - P\|_\cG &\le \frac{8\sqrt{2}J(\cG)\|F\|_{L^2(P)}}{\sqrt{N}} \nonumber\\
	&\le \frac{16\sqrt{6\pi R d}\|F\|_{L^2(P)}}{\sqrt{N}} \nonumber\\
	&= \frac{16L\sqrt{6\pi R d}((R \vee 1)+\|U\|_2)}{\sqrt{N}} \nonumber\\
	&\le \frac{16L\sqrt{6\pi R d}((R \vee 1)+\|U\|_{\psi_2})}{\sqrt{N}}\label{eq:expect_sup}
\end{align}
Furthermore, we estimate
\begin{align}
	\sigma_P(\cG) &\le \|F\|_{L^2(P)} \nonumber\\
	&\le \| F(Z) \|_{\psi_2} \nonumber\\
	&= \| L((R \vee 1) + U) \|_{\psi_2} \nonumber\\
	&\le L \left((R \vee 1) + \|U\|_{\psi_2}\right) \label{eq:sigmaPG}
\end{align}
and
\begin{align}
	\left\|\max_{j \le N} F(Z_j)\right\|_{\psi_2} &= L\left\| (R \vee 1) + \max_{j \le N} U_j \right\|_{\psi_2} \nonumber\\
	&\le L(R\vee 1) + 4L\sqrt{\log N} \|U\|_{\psi_2}, \label{eq:maxpsi2}
\end{align}
where we have used the triangle inequality for $\|\cdot\|_{\psi_2}$, as well as the maximal inequality \eqref{eq:maxineq}. Using the estimates \eqref{eq:expect_sup}, \eqref{eq:sigmaPG}, and \eqref{eq:maxpsi2} in Adamczak's inequality, we obtain \eqref{eq:gtail}.
\end{proof}

We are now ready to prove Theorem~\ref{thm:SAA}. The proof is via the probabilistic method. Let $\eps > 0$ and $R > 0$ be given, and choose
\begin{align*}
	N = \left\lceil \left(\frac{C\sqrt{d}}{\eps} \cdot L\left((R \vee 1) + \sqrt{d} + \sqrt{6}\right) \cdot \left( 16\sqrt{6\pi R d} + 5\sqrt{\log 4(d+1)}  \right) \right)^2 \right\rceil,
\end{align*}
where $C > 0$ is the absolute constant in the bound of Lemma~\ref{lm:app_SAA}.
	 Let $Z_1,\ldots,Z_N$ be i.i.d.\ copies of $Z \sim \gamma_d$, and observe that $\E[f(x+\sqrt{t}Z)] = Q_tf(x)$ and $\E[\partial_i f(x+\sqrt{t}Z)]=\partial_i Q_tf(x) = Q_t\partial_i f(x)$ for all $x \in \Reals^d$, $t \ge 0$, and $i \in [d]$. Define the events
	\begin{align*}
		E_0 &\deq \left\{ \max_{n \le N} \|Z_n\| \ge 8\sqrt{(d+6)\log N} \right\} \\
		E_1 &\deq \left\{ \sup_{x \in \sB^d(R)} \sup_{t \in [0,1]} \left|\frac{1}{N}\sum^N_{n=1}f(x+\sqrt{t}Z_n)-Q_tf(x)\right| \ge \eps \right\} \\
		E_2 &\deq \left\{ \max_{i \in [d]}\sup_{x \in \sB^d(R)} \sup_{t \in [0,1]} \left| \frac{1}{N}\sum^N_{n=1}\partial_i f(x+\sqrt{t}Z_n) - \partial_i Q_tf(x)\right| \ge \frac{\eps}{\sqrt{d}}\right\}.
	\end{align*}
We will show that $\PP\{E_0 \cup E_1 \cup E_2\} < 1$, which will imply that there exists at least one realization of $Z_1,\ldots,Z_N$ verifying the statement of the theorem.

By Lemma~\ref{lm:Gaussian_Orlicz_norm}, $U = \|Z\|$ satisfies $\|U\|_{\psi_2} \le \sqrt{d}+\sqrt{6}$, and therefore $U^*_N \deq \max_{n \le N} U_n$ satisfies $\|U^*_N\|_{\psi_2} \le \sqrt{32(d+6)\log N}$ by the maximal inequality \eqref{eq:maxineq}. Consequently, it follows from \eqref{eq:psi2_tail_bound} that
\begin{align*}
	\PP\{E_0\} &\le \PP\{U^*_N \ge \sqrt{2}\|U^*_N\|_{\psi_2}\} \le \frac{1}{e^2-1} \le \frac{1}{4}.
\end{align*}
Moreover, since the function $f$ and all of its partial derivatives are $L$-Lipschitz, Lemma~\ref{lm:app_SAA} (with $\gamma = \log 4(d+1)$) and the union bound give $\PP\{E_1 \cup E_2\} \le 1/4$. Therefore, $\PP\{E_0 \cup E_1 \cup E_2\} \le 1/2$.

\section{The proof of Theorem~\ref{thm:neural_net_apx}: uniform approximation of the F\"ollmer drift by a neural net}
\label{app:neural_net_apx_proof}

We first collect a few preliminaries.

\begin{lemma}[cheap gradient principle, \citet{griewank2008AD}]\label{lm:cheap_g} Let $f : \Reals^d \to \Reals$ be implementable by a neural net with differentiable activation function $\sigma : \Reals \to \Reals$, where the neural net has size (number of nodes) $m$ and depth (number of layers) $\ell$. Then each coordinate of the gradient $\nabla f$ can be computed by a neural net that has size $\cO(m + \ell)$, and where the activation function of each neuron is an element of the set $\{\sigma,\sigma'\}$.
\end{lemma}

\begin{lemma}[approximating multiplication and reciprocals]\label{lm:multi} Let $\sigma : \Reals \to \Reals$ be an activation function satisfying Assumption~\ref{as:universal_activation}. Then:
	\begin{enumerate}
		\item For any $M > 0$ and any $\delta > 0$, there exists a $2$-layer neural net $g : \Reals^2 \to \Reals$ of size $m \le 8c_\sigma \frac{M^2}{\delta} + 1$, such that 
		\begin{align}
			\sup_{x,y \in [-M,M]} |g(x,y)-xy| \le \delta.
		\end{align}
		\item For any $0 < a \le b < \infty$ and any $\delta > 0$, there exists a $2$-layer neural net $q : \Reals \to \Reals$ of size $m \le c_\sigma \frac{b}{a^2\delta} + 1$, such that
		\begin{align}
			\sup_{x \in [a,b]} \left|q(x)-\frac{1}{x}\right| \le \delta.
		\end{align}
	\end{enumerate}
\end{lemma}
\begin{remark} {\em These approximations suffice for our purposes. However, if one uses the ReLU activation function $x \mapsto x \vee 0$, then both multiplication and reciprocals can be $\eps$-approximated by neural nets with size and depth polylogarithmic in $1/\eps$ \citep{yarotsky2017relu,telgarsky2017rational}.}
\end{remark}
\begin{proof}
	For multiplication, we first consider the function $x \mapsto x^2 \wedge (4M^2)$, which is $4M$-Lipschitz and constant outside the interval $[-2M,2M]$. Assumption~\ref{as:universal_activation} then grants the existence of a univariate function $g_0 : \Reals \to \Reals$ of the form \eqref{eq:2layer} with $m \le 4c_\sigma \frac{M^2}{\delta}$ satisfying $|g_0(x)-x^2| \le 2\delta$ for all $x \in [-M,M]$. The desired approximation $g : \Reals^2 \to \Reals$ is given by
	\begin{align*}
		g(x,y) = \frac{1}{4}\left(g_0(x+y)-g_0(x-y)\right),
	\end{align*}
	which is a $2$-layer neural net with size $m \le 8c_\sigma \frac{M^2}{\delta} + 1$.
	Indeed, using the polarization identity $4xy = (x+y)^2-(x-y)^2$, we have
	\begin{align*}
		&\sup_{x,y \in [-M,M]} |g(x,y) - xy| \\
		&\le \frac{1}{4}\sup_{x,y \in [-M,M]} \left|g_0(x+y)-(x+y)^2\right| + \frac{1}{4}\sup_{x,y \in [-M,M]}\left|g_0(x-y)-(x-y)^2\right| \\
		&\le \delta.
	\end{align*}
	For approximating the reciprocal, consider the univarite function
	\begin{align*}
		x \mapsto \frac{1}{a} \1{\{x < a\}} + \frac{1}{x} \1{\{a \le x \le b\}} + \frac{1}{b} \1{\{x > b\}},
	\end{align*}
	which is $(1/a^2)$-Lipschitz and constant outside of the interval $[-b,b]$. The existence of the function $q$ with the stated properties follows immediately from  Assumption~\ref{as:universal_activation}.
\end{proof}

We now prove Theorem~\ref{thm:neural_net_apx}. Let $\delta = \frac{c^2\eps}{16L}$. By Theorem~\ref{thm:SAA}, there exist points $z_1,\ldots,z_N \in \Reals^d$ with $N = {\rm poly}(1/\delta,d,L,R)$, such that $R_{N,d} \deq \max_{n \le N}\|z_n\| \le 8\sqrt{(d+6)\log N}$, and the function $\varphi : \Reals^d \times [0,1] \to \Reals$ defined by
\begin{align*}
	\varphi(x,t) \deq \frac{1}{N}\sum^N_{n=1}f(x+tz_n)
\end{align*}
satisfies
\begin{align*}
	\sup_{x \in \sB^d(R)} \sup_{t \in [0,1]} |\varphi(x,\sqrt{t})-Q_t f(x)| \le \delta \quad \text{and} \quad
	\sup_{x \in \sB^d(R)} \sup_{t \in [0,1]}\|\nabla \varphi(x,\sqrt{t})-\nabla Q_t  f(x)\| \le \delta.
\end{align*}
By Assumption~\ref{as:f_nn_apx}, there exists a neural net $\wh{f} : \Reals^d \to \Reals$ be that approximates $f$ and the gradient of $f$ to accuracy $\delta$ on the blown-up ball $\sB^d(R+R_{N,d})$. Then the function 
\begin{align*}
	\wh{\varphi} : \Reals^d \times [0,1] \to \Reals, \qquad \wh{\varphi}(x,t) \deq \frac{1}{N}\sum^N_{n=1} \wh{f}(x+tz_n)
\end{align*}
can be computed by a neural net of size $N\cdot{\rm poly}(1/\delta,d,L,R)$, such that
\begin{align*}
&	\sup_{x \in \sB^d(R)}\sup_{t \in [0,1]}|\wh{\varphi}(x,\sqrt{t})-Q_t f(x)| \nonumber\\
&\le \sup_{x \in \sB^d(R)}\sup_{t \in [0,1]} |\wh{\varphi}(x,\sqrt{t})-\varphi(x,\sqrt{t})| + \sup_{x \in \sB^d(R)}\sup_{t \in [0,1]} |\varphi(x,\sqrt{t})-Q_tf(x)| \\
&\qquad \le \sup_{x \in \sB^d(R+R_{N,d})} |\wh{f}(x)-f(x)| + \sup_{x \in \sB^d(R)}\sup_{t \in [0,1]} |\varphi(x,\sqrt{t})-Q_tf(x)| \le 2\delta
\end{align*}
and
\begin{align*}
	&	\sup_{x \in \sB^d(R)}\sup_{t \in [0,1]}\|\nabla\wh{\varphi}(x,\sqrt{t})-\nabla Q_t  f(x)\| \nonumber\\
	&\le \sup_{x \in \sB^d(R)}\sup_{t \in [0,1]} \|\nabla \wh{\varphi}(x,\sqrt{t})-\nabla \varphi(x,\sqrt{t})\| + \sup_{x \in \sB^d(R)}\sup_{t \in [0,1]} \|\nabla\varphi(x,\sqrt{t})-\nabla Q_t f(x)\| \\
	&\qquad \le \sup_{x \in \sB^d(R+R_{N,d})} \|\nabla \wh{f}(x)-\nabla f(x)\| + \sup_{x \in \sB^d(R)}\sup_{t \in [0,1]} \|\nabla\varphi(x,\sqrt{t})-\nabla Q_t f(x)\| 
	\le 2\delta.
\end{align*}
Since $f$ is $L$-Lipschitz and bounded below by $c$, we have $c \le Q_tf(x) \le L(\|x\|+\sqrt{d})+f(0)$ for any $x \in \Reals^d$ and $t \in [0,1]$. Therefore, on $\sB^d(R) \times [0,1]$,
\begin{align*}
	\frac{c}{2} \le \wh{\varphi}(x,\sqrt{t}) \le L(R+\sqrt{d}) + f(0) + \frac{c}{2}
\end{align*}
where we have used the fact that $\delta \le c/4$. Without loss of generality, we may assume that $L \ge 1$. Then, for any  $x \in \sB^d(R)$ and $t \in [0,1]$,
\begin{align*}
& \left\| \nabla \log \wh{\varphi}(x,\sqrt{t}) - \nabla \log Q_t f(x) \right\| \\
&=	\left\| \frac{\nabla \wh{\varphi}(x,\sqrt{t})} {\wh{\varphi}(x,\sqrt{t})} - \frac{\nabla Q_tf(x)}{Q_tf(x)}\right\| \nonumber\\
&\le \frac{1}{\wh{\varphi}(x,\sqrt{t})} \| \nabla \wh{\varphi}(x,\sqrt{t}) - \nabla Q_t f(x)\| + \left\| \frac{\nabla Q_t f(x)}{Q_tf(x)} \right\| \frac{|\wh{\varphi}(x,\sqrt{t})-Q_tf(x)|}{\wh{\varphi}(x,\sqrt{t})} \\
&\le \frac{2L}{c} \cdot 2\delta + \frac{L}{c} \cdot \frac{2}{c} \cdot 2\delta \\
&\le \frac{\eps}{2},
\end{align*}
where we have used Lemma~\ref{lm:Follmer_drift_regularity} to bound $\| \frac{\nabla Q_t f}{Q_t f}\| \le L/c$. In other words, $\nabla \log \wh{\varphi}(x,\sqrt{t})$ approximates $\nabla \log Q_tf(x)$ to accuracy $\eps/2$ uniformly on $\sB^d(R) \times[0,1]$. It remains to approximate $\nabla \log \wh{\varphi}(x,\sqrt{t})$ by a neural net to accuracy $\eps/2$.

To that end, we first represent $\nabla \log \wh{\varphi}(x,\sqrt{t})$ as a composition of several elementary operations and then approximate each step by a neural net. Specifically, the computation of $v_i = \partial_i \log \wh{\varphi}(x,\sqrt{t})$ can be represented as a computation graph with the following structure:
\begin{enumerate}
	\item Compute $a = \wh{\varphi}(x,\sqrt{t})$.
	\item Compute $b_i = \partial_i\wh{\varphi}(x,\sqrt{t})$.
	\item Compute $r = 1/a$.
	\item Compute $v_i = rb_i$.
\end{enumerate}
Given $x$ and $\sqrt{t}$, $a$ is computed by a neural net with activation function $\sigma$, of size ${\rm poly}(1/\delta,d,L,R)$ and depth ${\rm poly}(1/\delta,d,L,R)$. Therefore, by the cheap gradient principle (Lemma~\ref{lm:cheap_g}), $b_i$ can be computed by a neural net of size ${\rm poly}(1/\delta,d,L,R)$, where the activation function of each neuron is an element of the set $\{\sigma,\sigma'\}$. Next, since $a$ takes values in $[c/2,L(R+\sqrt{d})+f(0)+c/2]$, by Lemma~\ref{lm:multi} the reciprocal $r = 1/a$ can be computed to accuracy $\eps/(4L\sqrt{d})$ by a $2$-layer neural net with activation function $\sigma$ and of size
$$
\cO\left(\frac{4}{c^2} \cdot \left(L(R+\sqrt{d})+f(0)+c/2\right) \cdot \frac{4L\sqrt{d}}{\eps}\right) \le {\rm poly}(1/\eps,d,L,R,c,1/c)
$$
Let $\wh{r}$ denote the resulting approximation. Then, since $|b_i| \le 2L$ and $|\wh{r}| \le 2/c + \eps/(4L\sqrt{d}) \le 4/c$, by Lemma~\ref{lm:multi} the product $\wh{r}b_i$ can be approximated to accuracy $\eps/4\sqrt{d}$ by a $2$-layer neural net with activation function $\sigma$ and with at most
$$
\cO\left((4/c \vee 2L)^2 \cdot \frac{4\sqrt{d}}{\eps}\right) \le {\rm poly}(1/\eps,d,L,1/c)
$$
neurons. The overall accuracy of approximation is
\begin{align*}
	\left|\wh{v}_i-v_i\right| &\le \left|\wh{v}_i - \wh{r}b_i\right| + \left|\wh{r}b_i - rb_i\right| 
	\le \frac{\eps}{2\sqrt{d}}.
\end{align*}
Thus, the vector $v = (v_1,\ldots,v_d)$ can be $\eps/2$-approximated by $\tilde{v}(x,\sqrt{t})$, where $\tilde{v} : \Reals^d \times [0,1] \to \Reals^d$ is a neural net with vector-valued output that has the size ${\rm poly}(1/\eps,d,L,R,c,1/c)$. Finally, since $\sup_{x \in \sB^d(R)}\sup_{t \in [0,1]}|\tilde{v}_i(x,\sqrt{t})| \le 2L/c$, the function
$$
\wh{v}_i(x,\sqrt{t}) \deq \min \{ \max \{ \tilde{v}_i(x,\sqrt{t}),-2L/c \}, 2L/c \}
$$
is continuous, takes values in $[-2L/c,2L/c]$ and coincides with $\tilde{v}_i$ on $\sB^d(R) \times [0,1]$. Moreover, the min and max operations can each be implemented exactly using $\cO(1)$ ReLU neurons.

\section{Proof of Theorem~\ref{thm:unbiased}}
\label{app:unbiased}

\subsection{Unbiasedness}

We follow the strategy of \citet{henry2017unbiased} and construct a sequence $\{\psi_n\}_{n \ge 0}$ of unbiased estimators, such that $\E[\psi_n] \xrightarrow{n \to \infty} \E[\psi]$, where $\psi \deq \lim_{n \to \infty}\psi_n$. By a standard approximation argument, we can assume that $g$ is bounded and Lipschitz. 

Let $\Delta^T_k \deq T_k - T_{k-1}$ and $\Delta^W_k \deq W_{T_k}-W_{T_{k-1}}$, for $k \ge 1$. For each $n \ge 0$, let
\begin{align}
	\psi_n &\deq g(\wh{X}_1) \cdot \frac{1}{1-F_\tau(\Delta^T_{n+1})} \prod^{N \wedge n}_{k=1} \frac{\left(b(\wh{X}_{T_k},T_k)-b(\wh{X}_{T_{k-1}},T_{k-1})\right)^\trn \Delta^W_{k+1}}{f_\tau(\Delta^T_k) \Delta^T_{k+1}} \cdot \1_{\{N \le n\}}\nonumber\\
	& \qquad + \prod^{n+1}_{k=1} \frac{1}{f_\tau(\Delta^T_k)} \cdot \left(b(\wh{X}_{T_{n+1}},T_{n+1})-b(\wh{X}_{T_n},T_n)\right)^\trn \nabla h(\wh{X}_{T_{n+1}},T_{n+1}) \cdot \frac{\Delta^W_{n+1}}{\Delta^T_{n+1}} \cdot \1_{\{N > n\}}, \label{eq:psi_n}
\end{align}
where $h(x,t) \deq \E[g(X_1)|X_t = x]$. We will show that $\E[\psi_n] = \E[g(X_1)]$ for all $n$ and that the sequence $\{\psi_n\}_{n \ge 0}$ is uniformly integrable. Then it will follow from the dominated convergence theorem that
\begin{align}\label{eq:psi}
	\psi = \lim_{n \to \infty} \psi_n = \frac{1}{1-F_\tau(1-T_N)} \cdot g(\wh{X}_1)  \cdot \prod^N_{k=1} \frac{1}{f_\tau(T_{k}-T_{k-1})}\wh{\cW}_k
\end{align}
is also an unbiased estimator. Observe that the estimator $\wh{\psi}$ defined in \eqref{eq:psi2_tail_bound} differs from $\psi$: instead of $g(\wh{X}_1)$, we have $g(\wh{X}_1)-g(\wh{X}_N)\1_{\{N > 0\}}$. Just as in \cite{henry2017unbiased}, the term proportional to $g(\wh{X}_N)\1_{\{N > 0\}}$ serves as a \textit{control variate} to ensure that $\wh{\psi}$ has finite variance. Indeed, since $\E[\Delta^W_{N+1}|T_N] = 0$, it is easy to see that
\begin{align*}
	\E\left[\frac{1}{1-F_\tau(1-T_N)} \cdot g(\wh{X}_N)\1_{\{N > 0\}}  \cdot \prod^N_{k=1} \frac{1}{f_\tau(T_{k}-T_{k-1})}\wh{\cW}_k\right] = 0,
\end{align*}
and therefore $\E[\wh{\psi}-\psi] = 0$.

Given $x,v \in \Reals^d$ and $t \in [0,1]$, consider the constant-drift diffusion process $\{\tilde{X}^{t,x,v}_s\}_{s \in [t,1]}$ with $\tilde{X}^{t,x,v}_t = x$ and
\begin{align*}
	\dif\tilde{X}^{t,x,v}_s = v\dif s + \dif W_s, \qquad s \in [t,1].
\end{align*}
This process has the infinitesimal generator
\begin{align*}
	\cL^v h(x,t) \deq v^\trn \nabla h(x,t) + \frac{1}{2}\tr \nabla^2 h(x,t), \qquad \forall h \in C^{2,1}(\Reals^d \times [0,1]).
\end{align*}
Then, by Dynkin's formula \citep[Lemma~19.21]{kallenberg2002prob_book}, for any $t \le s \le 1$, 
\begin{align}\label{eq:Dynkin}
	h(\tilde{X}^{t,x,v}_s,s) = h(\tilde{X}^{t,x,v}_t,t) + \int^s_t \left\{ \frac{\partial}{\partial r} + \cL^v\right\}h(\tilde{X}^{t,x,v}_r,r) \dif r + M^{t}_s,
\end{align}
where $\{M^{t}_s\}_{s \in [t,1]}$ is a martingale. In particular, let $h \in C^{2,1}(\Reals^d \times [0,1])$ be a bounded solution of the Cauchy problem
\begin{align}\label{eq:Cauchy_eq}
	\frac{\partial h}{\partial t} + \cL_t h = 0, \qquad h(\cdot,1) = g(\cdot)
\end{align}
where
\begin{align*}
	\cL_t h(x,t) \deq b(x,t)^\trn \nabla h(x,t) + \frac{1}{2} \tr \nabla^2 h(x,t).
\end{align*}
Rewriting \eqref{eq:Cauchy_eq} as
\begin{align*}
	\frac{\partial h}{\partial t} + \cL^v h = (v-b)^\trn h, \qquad h(\cdot,1) = g(\cdot)
\end{align*}
and using this in \eqref{eq:Dynkin}, we obtain the formula
\begin{align*}
&	h(\tilde{X}^{t,x,v}_s,s) \nonumber\\
& \qquad = g(\tilde{X}^{t,x,v}_1) + \int^1_s \big( b(\tilde{X}^{t,x,v}_r,r)-v\big)^\trn \nabla h(\tilde{X}^{t,x,v}_r,r) \dif r + M^{t}_s - M^{t}_1, \qquad t \le s \le 1.
\end{align*}
In particular, since $h(x,t) = \E[g(X_1)|X_t = x]$ by the Feynman--Kac formula, we have
\begin{align}\label{eq:h_FK}
	 h(x,t) = \E\left[g(\tilde{X}^{t,x,v}_1) + \int^1_t \big( b(\tilde{X}^{t,x,v}_s,s)-v\big)^\trn \nabla h(\tilde{X}^{t,x,v}_s,s) \dif s\right],
\end{align}
where $\E[M^{t}_t - M^{t}_1] = 0$ since $M^{h,t}$ is a martingale.

Using Eq.~\eqref{eq:h_FK} with $t = 0$ and $v = v_0 \deq b(x,0)$, we have
\begin{align*}
	h(x,0) = \E\left[g(\tilde{X}^{0,x,v_0}_1)+ \int^1_0 \big(b(\tilde{X}^{t,x,v}_s,s)-b(x,0)\big)^\trn \nabla h(\tilde{X}^{t,x,v}_s,s) \dif s \right].
\end{align*}
Recalling that $T_1 = \tau_1 \wedge 1$ is independent of the Brownian motion $\{W_t\}$ and $\PP[T_1 \ge 1] = \PP[\tau_1 \ge 1] = 1-F_\tau(1)$, we have
\begin{align}
	\E[g(\tilde{X}^{0,x,v_0}_1)] &= \frac{1}{1-F_\tau(1)} \E[g(\tilde{X}^{0,x,v_0}_1)\1_{\{T_1 \ge 1\}}], \label{eq:h_FK_term1}
\end{align}
and
\begin{align}
&	\E\left[\int^1_0 \big(b(\tilde{X}^{0,x,v_0}_s,s)-b(x,0)\big)^\trn \nabla h(\tilde{X}^{0,x,v_0}_s,s) \dif s \right] \nonumber\\
&\qquad = \E\left[ \frac{1}{f_\tau(T_1)}\big(b(\tilde{X}^{0,x,v_0}_{T_1},T_1)-b(x,0)\big)^\trn \nabla h(\tilde{X}^{0,x,v_0}_{T_1},T_1) \1_{\{T_1 < 1\}} \right]  \label{eq:h_FK_term2}.
\end{align}
Since the process $\tilde{X}^{0,x,v_0}$ coincides with $\wh{X}$ on $[0,T_1]$, it follows from \eqref{eq:h_FK_term1} and \eqref{eq:h_FK_term2} that
\begin{align}
	& h(x,0) \nonumber\\
	&= \E\left[\frac{1}{1-F_\tau(\Delta^T_1)} g(\wh{X}_1)\1_{\{T_1 \ge 1\}} +  \frac{1}{f_\tau(\Delta^T_1)}\big(b(\wh{X}_{T_1},T_1)-b(\wh{X}_{T_0},T_0)\big)^\trn \nabla h(\wh{X}_{T_1},T_1) \1_{\{T_1 < 1\}}\right] \label{eq:V0_a}\\
	&= \E[\psi_0], \nonumber
\end{align}
where the last equality follows from the fact that $T_1 = \Delta^T_1 \ge 1$ if and only if $N = 0$.

By Lemma~\ref{lm:Gaussian_by_parts} in Section~\ref{app:aux},
\begin{align}
	&\nabla h(x,0) \nonumber\\
	&= \E\left[g(\tilde{X}^{0,x,v_0}_1) W_1 + \int^1_0 \left(\big( b(\tilde{X}^{t,x,v}_s,s)-b(x,0)\big)^\trn \nabla h(\tilde{X}^{t,x,v}_s,s)\right) \frac{W_s}{s} \dif s \right] \nonumber\\
	&= \E\Bigg[\frac{1}{1-F_\tau(1)}g(\wh{X}_1) \frac{\Delta^W_1}{\Delta^T_1} \1_{\{T_1 \ge 1\}} \nonumber\\
	& \qquad \qquad + \frac{1}{f_\tau(\Delta^T_1)}\left(\big( b(\wh{X}_{T_1},T_1)-b(\wh{X}_{T_0},T_0)\big)^\trn \nabla h(\wh{X}_{T_1},T_1)\right) \frac{\Delta^W_1}{\Delta^T_1} \1_{\{T_1 < 1\}} \Bigg] \label{eq:grad_h}
\end{align}
Moreover, if we change the initial condition from $t=0,v=v_0$ to $t=T_1,v=v_1 \deq b(\wh{X}_{T_1},T_1)$, then it follows from \eqref{eq:grad_h} that, conditionally on $(\wh{X}_{T_1},T_1)$, whenever $T_1 < 1$,
\begin{align}
	\nabla h(\wh{X}_{T_1},T_1) &= \E\Bigg[ \frac{1}{1-F_\tau(\Delta^T_2)}g(\wh{X}_1) \frac{\Delta^W_2}{\Delta^T_2}\1_{\{T_2 \ge 1\}} \nonumber\\
	&  \qquad + \frac{1}{f_\tau(\Delta^T_2)}\left(\big( b(\wh{X}_{T_2},T_2)-b(\wh{X}_{T_1},T_1)\big)^\trn \nabla h(\wh{X}_{T_2},T_2)\right) \frac{\Delta^W_2}{\Delta^T_2} \1_{\{T_2 < 1\}} \Bigg| \wh{X}_{T_1},T_1 \Bigg]. \label{eq:grad_h_T1}
\end{align}
Substituting \eqref{eq:grad_h_T1} into \eqref{eq:V0_a} and using the fact that the event $\{T_1 < 1 \le T_2\}$ is equivalent to $\{N=1\}$, we have $h(x,0) = \E[\psi_1]$. Repeating this procedure, we have
\begin{align*}
	\E[g(X_1)|X_0 = x] = h(x,0) = \E[\psi_n], \qquad n \ge 0.
\end{align*}
We claim that the sequence $\{\psi_n\}_{n \ge 0}$ is uniformly integrable. To see this, first observe that, for each $k$, $\E[\|\Delta^W_{k+1}\|||T_{k+1}] \le (\Delta^T_{k+1}d)^{1/2}$. Then the uniform integrability follows from the boundedness of $b$, $g$, $\nabla h$, and from Lemma~\ref{lm:integrability} in Section~\ref{app:aux}. Therefore, taking the limit as $n \to \infty$, we obtain
\begin{align*}
	\E[g(X_1)|X_0 = x] = \lim_{n \to \infty} \E[\psi_n] = \E\left[\lim_{n \to \infty}\psi_n\right] = \E[\psi],
\end{align*}
where the second equality follows from the dominated convergence theorem.

\subsection{Variance}

Let $L \deq L_b \vee L_g$. For $1 \le k \le N+1$, let $\Delta^{\wh{X}}_k \deq \wh{X}_{T_{k+1}}-\wh{X}_{T_k}$ denote the increments of $\wh{X}$. Since $T_{N+1}=1$, we have
\begin{align*}
	\left|g(\wh{X}_1)-g(\wh{X}_{T_N})\1_{\{N > 0\}}\right| &\le \begin{cases}
	|g(x)| + L|\Delta^{\wh{X}}_1|, & N = 0 \\
	L|\Delta^{\wh{X}}_{N+1}|, & N > 0
\end{cases}
\end{align*}
which gives
\begin{align*}
	\left|g(\wh{X}_1)-g(\wh{X}_{T_N})\1_{\{N > 0\}}\right| &\le |g(x)|\1_{\{N = 0\}} + L\left(\sqrt{\Delta^T_{N+1}} + \|\Delta^{\wh{X}}_{N+1}\|\right).
\end{align*}
Using this and \eqref{eq:tau_pdf}, we can upper-bound $\wh{\psi}$ as follows:
\begin{align*}
	|\wh{\psi}| &\le \frac{e^a}{1-F_\tau(1)} \cdot \left( |g(x)|+ L\Big(\sqrt{\Delta^T_1} + \|\Delta^{\wh{X}}_1\|\Big)\right) \cdot\prod^N_{k=1} \frac{CL\Big(\sqrt{\Delta^T_{k+1}}+\|\Delta^{\wh{X}}_{k+1}\|\Big)}{\Delta^T_{k+1}} \cdot \|\Delta^W_{k+1}\|,
\end{align*}
where, for $k \ge 0$,
\begin{align*}
	\| \Delta^{\wh{X}}_{k+1} \| &= \| b(\wh{X}_{T_k},T_k) \cdot \Delta^T_{k+1} + \Delta^W_{k+1} \| \\
	&\le b_\infty \Delta^T_{k+1} + \|\Delta^W_{k+1}\|.
\end{align*}
Let $\cF_k \deq \sigma(T_j, \wh{X}_j : 1 \le j \le k)$. Then, since ${\rm Law}(\Delta^W_{k+1}|\cF_k) = {\rm Law}((\Delta^T_{k+1})^{1/2}Z|\cF_k)$, where $Z \sim \gamma_d$ is independent of $\cF_k \vee \sigma(T_{k+1})$, we have
\begin{align*}
&\E\Bigg[\Bigg(\frac{(\Delta^T_{k+1})^{1/2}+\|\Delta^{\wh{X}}_{k+1}\|}{\Delta^T_{k+1}} \cdot \|\Delta^W_{k+1}\|\Bigg)^2 \Bigg|\cF_k \Bigg] \nonumber\\
&\qquad\le \E_k\Bigg[\Bigg(\frac{b_\infty \Delta^T_{k+1} + (\Delta^T_{k+1})^{1/2} (1+\|Z\|)}{\Delta^T_{k+1}} \cdot \sqrt{\Delta^T_{k+1}}\|Z\|\Bigg)^2 \Bigg|\cF_k\Bigg] \\
&\qquad\le \E\left[(1+b_\infty + \|Z\|)^2 \|Z\|^2 \right] \\
&\qquad=:\kappa.
\end{align*}
Therefore, we can estimate
\begin{align*}
	\E[\wh{\psi}^2] &\le \left(\frac{e^a}{1-F_\tau(1)}\right)^2 \cdot \E\left[\left(|g(x)| + L(1+\sqrt{d})\right)^2\right] \cdot \E\left[\exp(\kappa N)\right].
\end{align*}

\subsection{Auxiliary lemmas}
\label{app:aux}

The following lemma is a straightforward consequence of the Gaussian integration-by-parts formula $\nabla_x \E[f(x+Z)] = \E[f(x+Z)Z]$, $Z \sim \gamma_d$, for any $C^1$ function $f: \Reals^d \to \Reals$:

\begin{lemma}[\citet{henry2017unbiased}]\label{lm:Gaussian_by_parts} Let $\nu$ be a positive measure on $[0,1]$. Let $\varphi : \Reals^d \times [0,1] \to \Reals$ be a continuous function, such that
	\begin{align*}
		\int^1_0 \E\left[\left\|\varphi(x + vt + W_t)\frac{W_t}{t}\right\|\right] \nu(\dif t) < \infty.
	\end{align*}
Then
\begin{align*}
	\nabla_x \left(\int^1_0 \E[\varphi(x+vt+W_t)]\nu(\dif t)\right) = \int^1_0 \E\left[\varphi(x+vt+W_t) \frac{W_t}{t}\right]\nu(\dif t)
\end{align*}
\end{lemma}
The next lemma is used to show that the sequence $\{\psi_n\}$ is uniformly integrable:

\begin{lemma}\label{lm:integrability} For any $C > 0$,
	\begin{align}
	\E\left[\frac{C^N}{1-F_\tau(\Delta^T_{N+1})}\prod^N_{k=1} \frac{1}{f_\tau(\Delta^T_k)(\Delta^T_{k+1})^{1/2}}\right] < \infty.
\end{align}
\end{lemma}
\begin{proof} For each $n \ge 0$, define the $n$-simplex
	$$
	\cS^n \deq \left\{ (s_1,s_2,\ldots,s_n) \in [0,1]^n : 0 < s_1 < \ldots < s_n < 1\right\}
	$$
with $s_0 \equiv 0$ and $s_{n+1} \equiv 1$. Consider the partial sums $S_k \deq \sum^k_{i=1}\tau_i$. Since the $\tau_i$'s are i.i.d., the conditional joint density of $(S_1,S_2,\ldots,S_n)$ given $N=n$ is equal to
	\begin{align*}
		q_n(s_1,s_2,\ldots,s_n) = \frac{1}{\PP[N=n]} \cdot (1-F_\tau(1-s_n)) \cdot \prod^n_{k=1}f_\tau(s_k - s_{k-1}), \qquad (s_1,\ldots,s_n) \in \cS^n
	\end{align*}
	where we have set $s_0 \equiv 0$. Then a calculation similar to the one in  Appendix~B of \citet{andersson2017unbiased} leads to
	\begin{align*}
		&\E\left[\frac{C^N}{1-F_\tau(\Delta^T_{N+1})}\prod^N_{k=1} \frac{1}{f_\tau(\Delta^T_k)(\Delta^T_{k+1})^{1/2}}\right] \nonumber\\
		&= \sum_{n \ge 0} \PP[N=n]\cdot C^n \int_{\cS^n} \frac{1}{1-F_\tau(s_n)} \prod^n_{k=1} \frac{1}{f_\tau(s_k - s_{k-1}) (s_{k+1} - s_{k})^{1/2}} q_n(s_1,\ldots,s_n) \dif s \\
		&\le \sum_{n \ge 0} C^n\int_{\cS^n} \prod^n_{k=0} \frac{1}{(s_{k+1} - s_{k})^{1/2}} \dif s \\
		&= \sqrt{\pi} \cdot E_{1/2,1/2}(C\sqrt{\pi}),
	\end{align*}
where $\dif s$ is the Lebesgue measure on $\cS^n$ and
\begin{align}\label{eq:Mittag_Leffler}
	E_{\alpha,\beta}(z) \deq \sum^\infty_{k=0} \frac{z^k}{\Gamma(\beta + \alpha k)}, \qquad z \in {\mathbb C},\, \alpha,\beta > 0
\end{align}
is the Mittag--Leffler function \citep{erdelyi_book}. When $\alpha$ and $\beta$ are both real and positive, the series in \eqref{eq:Mittag_Leffler} converges for all values of $z \in {\mathbb C}$, which completes the proof.
\end{proof}

\section{Proof of Lemma~\ref{lm:N_mgf}}

For each $t \ge 0$, let $N_t \deq \max\{ k : S_k < t \le S_{k+1}\}$. Then $N_1 = N$ and $T_n = S_n$ for $n \le N$. Moreover, $\{N_t\}_{t \ge 0}$ is a renewal process with renewal times $\{S_k\}_{k \ge 0}$ and i.i.d.\ interrenewal times with pdf $f_\tau$. The moment-generating function of $M_t$ can be upper-bounded as follows \citep{glynn1994counting}:
\begin{align}\label{eq:Nt_mgf}
	\E[e^{\theta N_t}] \le 1 + e^\theta \sum^\infty_{k = 0} e^{\theta k} \PP[S_k < t].
\end{align}
Let $t = 1$ and fix some $\beta > 0$. Then
\begin{align*}
	\sum^\infty_{k = 0} e^{\theta k} \PP[S_k < 1] &= \sum_{k \le \beta}e^{\theta k} \PP[S_k < 1] + \sum_{k > \beta} e^{\theta k}\PP[S_k < 1] \\
	&\le (\beta+1) e^{\theta \beta} + \sum^\infty_{k=0} e^{\theta k}\PP[S_k < k\beta^{-1}]. 
\end{align*}
Using Markov's inequality and the fact that the $\tau_i$'s are i.i.d., we can further estimate
\begin{align*}
	\PP[S_k < k\beta^{-1}] &= \PP[k - \beta S_k > 0] \le e^k \E\left[e^{-\beta S_k}\right] = \Big(e M_\tau(-\beta)\Big)^k.
\end{align*}
Substituting these estimates into \eqref{eq:Nt_mgf} and optimizing over $\beta$, we get \eqref{eq:N_mgf}.

\section*{Acknowledgments}

The authors would like to thank Matus Telgarsky for many enlightening discussions. This work  was supported in part by the NSF CAREER award CCF-1254041, in part by the Center for Science of Information (CSoI), an NSF Science and Technology Center, under grant agreement CCF-0939370, in part by the Center for Advanced Electronics through Machine Learning (CAEML) I/UCRC award no.~CNS-16-24811, and in part by the Office of Naval Research under grant no.~N00014-12-1-0998.
	
	\bibliography{ref.bbl}

    \end{document}